\documentclass{amsart}
\usepackage{amsfonts,amssymb,amsmath,amsthm}
\usepackage{url}
\usepackage{enumerate}
\usepackage{amssymb}
\usepackage{color}
\usepackage{verbatim} 

\newtheorem{thrm}{Theorem}[section]
\newtheorem{lem}[thrm]{Lemma}
\newtheorem{prop}[thrm]{Proposition}
\newtheorem{cor}[thrm]{Corollary}
\theoremstyle{definition}
\newtheorem{definition}[thrm]{Definition}

\newtheorem{Example}[thrm]{Example}

\numberwithin{equation}{section}
\renewcommand{\theequation}{\arabic{section}.\arabic{equation}}
\renewcommand{\theequation}{\thesubsection.\arabic{equation}}

\def\a1s{a_1,\cdots, a_s}
\def\a{\alpha}

\def\aa{\mathcal A}

\def\andd{\quad\hbox{and}\quad}

\def\b{\beta}

\def\bl4{B_{\ell\geq4}}

\def\d{\delta}

\def\bbbf{\mathbb{F}}

\def\gg{{\mathcal G}}

\def\fg{\mathfrak{g}}

\def\hh{{\mathcal H}}

\def\fh{\mathfrak{h}}

\def\lam{\lambda}
\def\Lam{\Lambda}
\def\LL{\mathcal{L}}

\def\fl{\mathfrak{L}}

\def\ep{\epsilon}
\def\fm{(\cdot,\cdot)}

\def\bbbq{\mathbb{Q}}

\def\m{\mathcal{M}}

\def\1k{\frac{1}{k}}
\def\op{\oplus}
\def\ot{\otimes}

\def\la{\langle}
\def\ra{\rangle}

\def\sub{\subseteq}
\def\sg{\sigma}

\def\rcross{R^{\times}}

\def\pf{\noindent{\bf Proof. }}

\def\fp{\mathfrak{p}}

\def\v{{\mathcal V}}

\def\w{{\mathcal W}}

\def\bbbz{{\mathbb Z}}

\def\1il{1\leq i\leq\ell}

\def\rre{R_{re}}
\def\rim{R_{ns}}

\begin{document}

\title[Extended Affine Lie Superalgebras]{Extended Affine Lie Superalgebras}

\maketitle

\centerline{\bf Malihe Yousofzadeh\footnote{
 Department of Mathematics, University of Isfahan, Isfahan, Iran,
and School of Mathematics, Institute for Research in
Fundamental
Sciences (IPM), Tehran, Iran.
}}

\bigskip

\parbox{5in}{
Abstract. We introduce the notion of extended affine Lie superalgebras and investigate the properties of their root systems.  Extended affine Lie algebras, invariant affine reflection algebras, finite dimensional basic classical simple Lie superalgebras   and affine Lie superalgebras are examples of extended affine Lie superalgebras.}

\section{Introduction}  Given an arbitrary $n\times n$-matrix $A$ and a subset $\tau\sub \{1,\ldots,n\},$ one can define the contragredient  Lie superalgebra $\gg(A,\tau)$ which is presented by a finite set of generators subject to specific relations. Contragredient Lie superalgebras associated with so-called  generalized Cartan matrices are  known as Kac-Moody Lie superalgebras. These Lie superalgebras  are of great importance among contragredient Lie superalgebras; in particular,  affine Lie superalgebras, i.e., those  Kac-Moody Lie superalgebras which are  of finite growth, but not of finite dimension and  equipped with a nondegenerate invariant even supersymmetric  bilinear form, play a significant  role in the theory of Lie superalgebras.
In the past 40 years, researchers in many areas of mathematics and mathematical physics have been attracted to Kac-Moody Lie superalgebras $\gg(A,\emptyset)$ known as Kac-Moody Lie algebras. These Lie algebras are a natural generalization of finite dimensional simple Lie algebras. One of the differences between affine Lie superalgebras and affine Lie algebras is the existence of nonsingular roots i.e., roots which are orthogonal to themselves   but not to all other roots. In 1990, R. H\o egh-Krohn and B. Torresani  \cite{HT} introduced irreducible  quasi simple Lie algebras  as a generalization of both affine Lie algebras and finite dimensional simple Lie algebras over complex numbers. In 1997, the authors in   \cite{AABGP} systematically studied irreducible  quasi simple Lie algebras under the name extended affine Lie algebras.
The existence of isotropic roots, i.e., roots which are orthogonal to all other roots,  is one of the phenomena which  occurs in  extended affine Lie algebras but not in finite dimensional simple Lie algebras.
Since 1997, different generalizations of extended affine Lie algebras have been  studied; toral type extended affine Lie algebras \cite{AKY}, locally  extended affine Lie algebras \cite{MY} and invariant affine reflection algebras \cite{N1}, as a generalization of the last two stated classes,  are examples of these generalizations.

Basic classical simple Lie superalgebras, orthosymplectic Lie superalgebras of arbitrary dimension as well as specific  extensions of particular root graded Lie superalgebras satisfy certain  properties which are in fact the super version of the axioms defining  invariant affine reflection algebras. In the present work, we  study the class of Lie superalgebras satisfying  these certain properties; we introduce the notion of extended affine Lie superalgebras. Roughly speaking, an extended affine Lie superalgebra is a Lie superalgebra having a weight space decomposition with respect to a nontrivial abelian subalgebra of the even part and equipped with a nondegenerate invariant even supersymmetric  bilinear form such that the weight vectors associated with  so-called real roots are ad-nilpotent.
We prove that the even part of an extended affine Lie superalgebra is an invariant affine reflection algebra.    We show that  corresponding to  each nonisotropic root $\a$ of an extended affine Lie superalgebra $\LL,$ there exists a triple of elements of $\LL$ generating a subsuperalgebra $\gg(\a)$ isomorphic to  either $\frak{sl}_2$ or $\frak{osp}(1,2)$ depending on whether  $\a$ is even or not. Considering $\LL$ as a $\gg(\a)$-module, we can  derive    some   properties of the corresponding root system of $\LL$ which are in fact  the features  defining  extended affine root supersystems \cite{you6}. As $\frak{osp}(1,2)$-modules  are important in the theory of extended affine Lie superalgebras, we devote a section to study the module theory of $\frak{osp}(1,2).$ Although, it is an old well-known fact that finite dimensional  $\frak{osp}(1,2n)$-modules are completely reducible, using the generic features of $\frak{osp}(1,2),$ we prove that finite dimensional  $\frak{osp}(1,2)$-modules are completely reducible in  a different approach  from the  one in the literature.   We conclude the paper with  some examples showing that starting form an extended affine Lie superalgebra, one can get a new one using an affinization process.

\renewcommand{\theequation}{\thesection.\arabic{equation}}
\section{Finite Dimensional Modules of $\frak{osp}(1,2)$}
Throughout this work, $\bbbf$ is a field of characteristic zero. Unless otherwise mentioned, all vector spaces are considered over $\bbbf.$ We denote the dual space of a vector space $V$ by $V^*.$ If $V$ is a vector space graded by an abelian group, we denote the degree of a homogeneous element $x\in V$ by $|x|;$ we also  make a convention that if  $|x|$ is appeared in an expression, for an element $x$ of $V,$ by default, we assume that $x$ is homogeneous.   If $X$ is a subset of a group $A,$ by $\la X\ra,$ we mean the subgroup of $A$ generated by $X.$ Also we denote the cardinal number of a  set $S$ by $|S|;$ and for two symbols $i,j,$ by $\d_{i,j},$ we mean the Kronecker delta. For a map $f:A\longrightarrow B$ and $C\sub A,$ by $f\mid_{_C},$ we mean the restriction of $f$ to $C.$ Also we use $\uplus$ to indicate the disjoint union.

In  the present paper, by a module of a Lie superalgebra $\fg,$ we mean a superspace $\v=\v_{\bar 0}\op \v_{\bar 1}$ and a bilinear map $\cdot: \fg\times \v\longrightarrow \v$ satisfying $\fg_{\bar i}\cdot \v_{\bar j}\sub \v_{\bar i+\bar j}$ for $i,j\in\{0,1\}$ and $[x,y]\cdot v=x\cdot (y\cdot v)-(-1)^{|x||y|}y\cdot (x\cdot v)$ for all $x,y\in\fg,$ $v\in\v.$ Also by a $\fg$-module homomorphism from a $\fg$-module $\v$ to a  $\fg$-module $\w,$ we mean a linear homomorphism $ \phi$ of parity $\bar i$ ($i\in\{0,1\}$) with $\phi(x\cdot v)=(-1)^{|x||\phi|}x\cdot \phi(v)$ for $x\in \fg,v\in \v.$

Also by a {\it symmetric form} on an additive abelian group $A,$ we mean a map $\fm: A\times A\longrightarrow \bbbf$ satisfying
\begin{itemize}
\item $(a,b)=(b,a)$ for all $a,b\in A,$
\item $(a+b,c)=(a,c)+(b,c)$ and $(a,b+c)=(a,b)+(a,c)$ for all $a,b,c\in A.$
\end{itemize}
In this case, we set  $A^0:=\{a\in A\mid(a,A)=\{0\}\}$ and call it the {\it radical} of the form $\fm.$ The form is called {\it nondegenerate} if $A^0=\{0\}.$
We note that if the form is nondegenerate, $A$ is torsion free and we can identify $A$ as a subset of $\bbbq\ot_\bbbz A.$ If $A$ is a vector space over $\bbbf,$ bilinear forms are used in the usual sense.
\

We recall that $\mathfrak{osp}(1,2)$ is a subsuperalgebra of $\frak{sl}(1,2)$ for which
{\small$$
F_+:=\left(
\begin{array}{ccc}
  0 & 1 & 0 \\
  0 & 0 & 0 \\
  1 &0 & 0 \\
\end{array}
\right),
F_-:=\left(
\begin{array}{ccc}
  0 & 0 &2\\
   -2 & 0 & 0 \\
 0 & 0 & 0 \\
\end{array}
\right),  $$}  of parity one, together with  {\small $$H:=\left(
\begin{array}{ccc}
 0 & 0 & 0 \\
 0& -2 & 0 \\
 0& 0 &2 \\
 \end{array}
\right),E_+=\left(
\begin{array}{ccc}
  0 & 0 & 0 \\
  0 & 0 & 0 \\
  0 & 2 & 0 \\
\end{array}
\right), E_-=\left(
\begin{array}{ccc}
0 & 0 & 0 \\
  0 &   0&-8 \\
  0 & 0 & 0 \\
  \end{array}
\right),$$}
 of parity zero, form a basis. The triple  $(F_+,F_-,H)$ is an $\mathfrak{osp}$-triple for $\mathfrak{osp}(1,2)$ in the following sense:
\begin{definition}{\rm
Suppose that $\fg=\fg_{\bar 0}\op\fg_{\bar 1}$ is a Lie superalgebra. We call a triple $(x,y,h)$ of nonzero elements of $\fg$ an {\it $\mathfrak{sl}_2$-super triple} for $\fg$ if   \begin{itemize}
\item $\{x,y,h\}$ generates the Lie superalgebra $\fg,$
\item  $x,y$ are homogenous of the same degree,
\item $[h,x]=2x,\; [h,y]=-2y,\; [x,y]=h.$
\end{itemize}}If $x,y\in \fg_{\bar 1},$ we refer to $(x,y,h)$ as an {\it $\frak{osp}$-triple} and note that if $x,y\in\fg_{\bar 0},$ $(x,y,h)$ is an $\mathfrak{sl}_2$-triple.
\end{definition}

\begin{lem}\label{iso} Suppose that  $(x,y,h)$ is an $\mathfrak{osp}$-triple for a Lie superalgebra $\fg=\fg_{\bar 0}\op\fg_{\bar 1},$ then $(\frac{1}{4}[x,x],-\frac{1}{4}[y,y],\frac{1}{2}h)$ is an $\mathfrak{sl}_2$-triple for $\fg_{\bar 0}$ and $\fg\simeq\frak{osp}(1,2).$
\end{lem}
\pf We have
$$\begin{array}{l}\;[[x,x],[y,y]]=-8h,\;\;[h,[x,x]]=4[x,x],\;\;[h,[y,y]]=-4[y,y],\\\\
\;[[x,x],x]=0,\;\;[[y,y],y]=0.
\end{array}
$$
Therefore, we get that $(\frac{1}{4}[x,x],-\frac{1}{4}[y,y],\frac{1}{2}h)$ is an $\frak{sl}_2$-triple; in particular $[x,x]\neq0$ as well as  $[y,y]\neq0$ and we have  $\fg=\bbbf [x,x]\op\bbbf [y,y]\op\bbbf h\op\bbbf y\op\bbbf x.$ Now it follows that $\fg$ is   isomorphic to $\mathfrak{osp}(1,2).$
\qed

\begin{lem}\label{gen}
Suppose that $(e,f,h)$ is an $\frak{osp}$-triple for a Lie superalgebra $\fg.$ Assume  $(\v,\cdot)$ is a $\fg$-module with corresponding representation $\pi.$ If $\lam\in\bbbf\setminus\{-2\}$ and $u\in \v_{\bar i}$ $(i\in\{0,1\})$ are such that  $h\cdot u=\lam u$ and  $[e,e]\cdot u=0,$  then  $\fg_{\bar 0}$-submodule of $\v$ generated by $f\cdot u$ equals to $$T:=\sum_{k\in \bbbz^{\geq0}}\bbbf f^{2k}\cdot (e\cdot u)+ \sum_{k\in \bbbz^{\geq0}}\bbbf f^{2k+1}\cdot(f\cdot(e\cdot u)-(\lam+2)u)$$ in which by the action of $f^k,$ we mean $\pi(f)^k$ for all $k\in\bbbz^{\geq 0}.$
\end{lem}
\pf
Since $h\cdot u=\lam u$ and  $e\cdot (e\cdot u)=\frac{1}{2}[e,e]\cdot u=0,$ we have $h\cdot(e\cdot u)=(\lam+2)e\cdot u,$ $h\cdot(f\cdot(e\cdot u))=\lam f\cdot(e\cdot u)$ and that
\begin{eqnarray*}
e\cdot(f\cdot(e\cdot u))=-f\cdot(e\cdot(e\cdot u))+h\cdot (e\cdot u)=(\lam+2)e\cdot u.
\end{eqnarray*}
 Now for $x\in\{f\cdot(e\cdot u)-(\lam+2)u,e\cdot u\},$ we have $e\cdot x=0$ and for $$\lam_x:=\left\{\begin{array}{ll}\lam& \hbox{if $x=f\cdot(e\cdot u)-(\lam+2)u$}\\
 \lam+2& \hbox{if $x=e\cdot u,$}\end{array}\right.$$ we have
$$\begin{array}{l}\;h\cdot (f^{k}\cdot x)=(\lam_x-2k)f^{k}\cdot x,\\
\;f\cdot (f^{k}\cdot x)=f^{k+1}\cdot x,\\\\
\;e\cdot (f^{k}\cdot x)=
\left\{\begin{array}{ll}
-k f^{k-1} \cdot x& \hbox{$k$ is even}\\
(\lam_x-(k-1))f^{k-1} \cdot x& \hbox{$k$ is odd}
\end{array}\right.
\end{array}$$ for $k\in\bbbz^{\geq 0},$ where $f^{-1}\cdot x$ is defined to be zero. Now it follows that $T$ is  invariant under the action of $[f,f],[e,e]$ and $h,$ i.e. it is  a $\fg_{\bar 0}$-submodule of $\v.$ On the other hand,  $$f\cdot u=\frac{1}{\lam+2}(f^2\cdot( e\cdot u)-f\cdot (f\cdot (e\cdot u-(\lam+2) u)))\in T.$$ Also if $S$ is a $\fg_{\bar 0}$-submodule of $\v$ containing $f\cdot u,$ then \begin{eqnarray*}
-2e\cdot u=-[h,e]\cdot u&=&-h\cdot (e\cdot u)+e\cdot (h\cdot u)\\
&=&-f\cdot (e\cdot (e\cdot u))-e\cdot (f\cdot (e\cdot u))+e\cdot (h\cdot u)\\
&=&-e\cdot (f\cdot (e\cdot u))+e\cdot (h\cdot u)\\
&=& e\cdot (e\cdot (f\cdot u))-e\cdot (h\cdot u)+e\cdot (h\cdot u)\\
&=& \frac{1}{2}[e,e]\cdot (f\cdot u)\in S
\end{eqnarray*}So $T\sub S.$ This completes the proof.
\qed

\begin{lem}\label{int eigen}
Suppose that $(\v,\cdot)$ is a finite dimensional  module for a Lie superalgebra $\fg\simeq\frak{osp}(1,2)$ with corresponding representation $\pi.$ Take $(e,f,h)$ to be an $\mathfrak{osp}$-triple for $\fg.$ Then we have the following:

(i) $\pi(h)$ is a diagonalizable  endomorphism of $\v$ with even integer eigenvalues each of which occurs with its opposite.

(ii) Suppose that  $\v$ is irreducible and $\Lam$ is the set of eigenvalues of $\pi(h).$ Then the corresponding eigenspaces are one-dimensional and  there is a nonnegative even integer $\lam$ with $\Lam=\{-\lam,-\lam+2,\ldots,\lam-2,\lam\}.$  Moreover, if $\lam\neq 0,$ $\v_{\bar 0}$ and $\v_{\bar 1}$ are irreducible $\fg_{\bar 0}$-submodules of $\v$ and  there is $i\in\{0,1\}$ such that  $\{-\frac{\lam}{2},-\frac{\lam}{2}+2,\ldots,\frac{\lam}{2}-2,\frac{\lam}{2}\}$ and $\{-\frac{\lam}{2}+1,-\frac{\lam}{2}+3,\ldots,\frac{\lam}{2}-3,\frac{\lam}{2}-1\}$ are  the set of eigenvalues of  $\frac{1}{2}\pi(h)|_{\v_{\bar i}}$ and $\frac{1}{2}\pi(h)|_{\v_{\bar i+\bar 1}}$ respectively.
\end{lem}
\pf
$(i)$ We know that $\fg_{\bar 0}\simeq \frak{sl}_2(\bbbf)$ and that  $(\frac{1}{4}[e,e],-\frac{1}{4}[f,f],\frac{1}{2}h)$ is an $\mathfrak{sl}_2$-triple for $\fg_{\bar 0}.$ Considering $\v$ as a $\fg_{\bar 0}$-module and using the $\frak{sl}_2$-module theory \cite[\S III.8]{J}, we get $\pi(\frac{1}{2}h)$ acts diagonally on $\v$ with integer eigenvalues each of which occurs with its opposite. This completes the proof.

$(ii)$ Suppose that $\v$ is irreducible. Take  $\lam$ to be  the largest eigenvalue of $\pi(h)$ and fix a  homogeneous  eigenvector $v_0$ for this eigenvalue. Set $v_{-1}:=0$ and   $v_i:=\pi(f)^i(v_0),$ for $i\in\bbbz^{\geq 0}.$
 For $i\in\bbbz^{\geq 0},$ we have $$h\cdot v_i=(\lam-2i)v_i,$$ $$f\cdot v_i=v_{i+1}$$ and $$e\cdot v_i=\left\{\begin{array}{ll}
-iv_{i-1} & \hbox{$i$ is even}\\
(\lam-(i-1))v_{i-1}& \hbox{$i$ is odd}.
\end{array}\right.$$ This together with the fact that $\v$ is irreducible  shows that  $\v=\sum_{k\in\bbbz^{\geq0}}\bbbf f^k\cdot v_0$ and so  each eigenspace is one-dimensional.
Since $\lam+2$ is not an eigenvalue,  by part $(i),$ $-\lam-2$ is not an eigenvalue; in particular, $v_{\lam+1}=0.$ So $\v=\sum_{k=0}^{\lam}\bbbf f^k\cdot v_0.$
This completes the proof if $\lam=0.$ Now suppose $\lam\neq 0.$ For $i\in\{0,\ldots,\lam\},$ we have  $v_i\in \v^{m_i}$ where $m_i:=\lam-2i$ and $$\v^{m_i}:=\{v\in \v\mid h\cdot v=m_iv\}=\{v\in \v\mid \frac{1}{2}h\cdot v=\frac{1}{2}m_iv\}.$$
Set $$U:=\hbox{span}_\bbbf\{v_{2i}\mid i\in\{0,\ldots, \frac{\lam}{2}\} \}\andd W:=\hbox{span}_\bbbf\{v_{2i+1}\mid i\in\{0,\ldots, \frac{\lam}{2}-1\} \}.$$
Both $U$ and $W$ are invariant under the actions of  $\frac{1}{4}[e,e],-\frac{1}{4}[f,f],\frac{1}{2}h$ and so they are  $\fg_{\bar 0}$-submodules of $\v.$ Since $\lam$ is the largest eigenvalue for $\pi(h)$ and $h\cdot( e\cdot v_0)=(\lam+2)v_0,$ we have $e\cdot v_0=0$ and so we get $0\neq \lam v_0=h\cdot  v_0=e\cdot f\cdot v_0+f\cdot e\cdot v_0=e\cdot f\cdot v_0.$ This implies that $f\cdot v_0\neq 0.$ So $U$ and $W$ are nonzero $\fg_{\bar 0}$-submodules.
For $i\in\{0,\ldots,\lam\},$ take  $U^{\frac{m_i}{2}}:=\v^{m_i}\cap U$ and $W^{\frac{m_i}{2}}:=\v^{m_i}\cap W.$ Then we have $$U=\sum_{i=0}^{\lam/2}\bbbf v_{2i}=U^{-\frac{\lam}{2}}\op U^{-\frac{\lam}{2}+2}\op\cdots\op U^{\frac{\lam}{2}-2}\op U^{\frac{\lam}{2}}$$ and $$W=\sum_{i=0}^{\lam/2-1}\bbbf v_{2i+1}=W^{-\frac{\lam}{2}+1}\op W^{-\frac{\lam}{2}+3}\op\cdots\op W^{\frac{\lam}{2}-3}\op W^{\frac{\lam}{2}-1}.$$  Using the  standard $\mathfrak{sl}_2$-module theory, we get that $v_i\neq 0$ for $0\leq i\leq \lam$ and that both $U$ and $W$ are irreducible $\fg_{\bar 0}$-modules. If the homogeneous element   $v_0$ is of degree $\bar i$ $(i\in \{0,1\}),$ we have   $\v_{\bar i}=U$ and $\v_{\bar i+\bar 1}=W.$ This completes the proof.
\qed
\begin{cor}\label{cor1} Each  (nonzero) finite dimensional irreducible $osp(1,2)$-module is of odd dimension. Moreover, suppose that $\lam$ is a nonnegative even integer and $\v$ is a  superspace with a basis $\{v_i\mid 0\leq i\leq \lam\}$ of homogeneous elements such that  $\{v_{2i}\mid 0\leq i\leq \frac{1}{2}\lam\}$ is a basis for either of  $\v_{\bar 0}$ or $\v_{\bar 1}.$ Take $(e,f,h)$ to be an  $\frak{osp}(1,2)$-triple for a superalgebra $\fg.$ Set $v_{\lam+1}=v_{-1}:=0$ and define $\cdot:\fg\times \v\longrightarrow \v$ by
$$\begin{array}{ll}
\begin{array}{l}
f\cdot v_i:=v_{i+1},\\
e\cdot v_i:=\left\{\begin{array}{ll}
-iv_{i-1} & \hbox{$i$ is even}\\
(\lam-(i-1))v_{i-1}& \hbox{$i$ is odd,}
\end{array}\right.\\
h\cdot v_i=(\lam-2i)v_i,
\end{array}& \begin{array}{l}
\;[f,f]\cdot v_i:=2f\cdot(f\cdot v_i),\\
\;[e,e]\cdot v_i:=2e\cdot(e\cdot v_i),\\
\end{array}
\end{array} $$
for $0\leq i\leq \lam.$ Then up to isomorphism, $(\v,\cdot)$ is the unique finite dimensional irreducible  $\fg$-module of dimension $\lam+1.$
\end{cor}
\begin{lem}\label{com red}
Each finite dimensional $\frak{osp}(1,2)$-module is completely reducible.
\end{lem}
\pf Fix an $\frak{osp}$-triple $(e,f,h)$ for $\fg:=\frak{osp}(1,2)$ and consider the $\frak{sl}_2$-triple $(\frac{1}{4}[e,e],-\frac{1}{4}[f,f],\frac{1}{2}h)$ for $\fg_{\bar 0}$ as in
Lemma \ref{iso}. Suppose that $\v$ is a finite dimensional $\frak{osp}(1,2)$-module with corresponding representation $\pi$. We know from Lemma \ref{int eigen} that  $\pi(h)$ is diagonalizable with even integer eigenvalues. We also know that $\v_{\bar 0}$ and $\v_{\bar 1}$ are both finite dimensional $\fg_{\bar 0}$-submodules of $\v.$ Suppose that $\v_{\bar 0}=\op_{j=1}^n W^j$ is a decomposition of $\v$ into irreducible $\fg_{\bar 0}$-modules. For each $1\leq j\leq n,$ take $w_j$ to be an eigenvector of the largest eigenvalue $\lam_j$ of $\pi(\frac{1}{2}h)$ on $W^j.$ We have \begin{equation}\label{final}
W^j=\op_{k=0}^{\lam_j} \bbbf [f,f]^k\cdot w_j=\op_{k=0}^{\lam_j} \bbbf f^{2k}\cdot w_j
\end{equation}
by the $\frak{sl}_2$-module theory. For $j\in\{1,\ldots,n\},$  set $$\displaystyle{ T^j:=\sum_{k=0}^{\infty}\bbbf f^{2k}\cdot (e\cdot w_j)}\andd S^j:=\displaystyle{\sum_{k=0}^{\infty}\bbbf f^{2k+1}\cdot(f\cdot (e\cdot w_j)-(2\lam_j+2)w_j)}.$$ We carry out the proof in the following steps:

\textbf{Step 1}. If $y$ is an eigenvector of $\pi(h)$ of  eigenvalue $2\lam$
such that $e\cdot y=0,$ then $ f^{2\lam+1}\cdot y=0:$ As for $k\in\bbbz^{\geq 0},$ $h\cdot(f^k\cdot y)=(\lam-2k)f^{k}\cdot y$ and $\v$ is finite dimensional, there is $k\in\bbbz^{\geq0}$ such that $f^k\cdot y\neq0$ but   $f^{k+1}\cdot y=0.$ Therefore, we have  $$0=e\cdot (f^{k+1}\cdot y)=\left\{\begin{array}{ll}
-(k+1) f^k\cdot y& \hbox{$k$ is odd}\\
(2\lam-k)f^{k}\cdot y& \hbox{$k$ is even.}
\end{array}\right.$$ This implies that $k=2\lam$ and so we are done.

\textbf{Step 2}. For each $j\in\{1,\ldots,n\},$  $$\displaystyle{T^j=\sum_{k=0}^{\lam_j+1}\bbbf f^{2k}\cdot (e\cdot w_j)}\andd \displaystyle{S^j=\sum_{k=0}^{\lam_j-1}\bbbf f^{2k+1}\cdot(f\cdot (e\cdot w_j)-(2\lam_j+2)w_j):}
$$
Since $w_j$ is an eigenvector of $\pi(h)$ with eigenvalue $2\lam_j,$ we have $h\cdot(e\cdot w_j)=(2\lam_j+2)(e\cdot w_j)$ and $h\cdot(f\cdot(e\cdot w_j))=2\lam_j f\cdot(e\cdot w_j).$ Also since $[e,e]\cdot w_j=0,$ as before, we have $$e\cdot(e\cdot w_j)=0\andd e\cdot (f\cdot (e\cdot w_j)-(2\lam_j+2)w_j)=0 $$ and so we are done using Step 1.

\textbf{Step 3}. {\small $\displaystyle{\v_{\bar 0}=\sum_{j=1}^n\sum_{k=0}^{\lam_j}\bbbf f^{2k+1}\cdot (e\cdot w_j)+\sum_{j=1}^n\sum_{k=0}^{\lam_j}\bbbf f^{2k}\cdot(f\cdot (e\cdot w_j)-(2\lam_j+2)w_j)}:$} Take $X$ to be the right hand side of this equality.
For each $j\in\{1,\ldots,n\},$ $$w_j=\frac{1}{2\lam_j+2}(f\cdot (e\cdot w_j)-(f\cdot (e\cdot w_j)-(2\lam_j+2)w_j)).$$ Therefore $w_j\in \displaystyle{\sum_{k=0}^{\lam_j}\bbbf f^{2k+1}\cdot (e\cdot w_j)+\sum_{k=0}^{\lam_j}\bbbf f^{2k}\cdot(f\cdot (e\cdot w_j)-(2\lam_j+2)w_j)\sub X}$
for all $j\in\{1,\ldots,n\}.$ This completes the proof as $X\sub \v_{\bar 0}$ and  $\{w_j\mid 1\leq j\leq n\}$ is a set of generators for the $\fg_{\bar 0}$-module $\v_{\bar 0}.$

\textbf{Step 4}.
For each $j\in\{1,\ldots,n\},$ set
$U^j$ to be the $\fg_{\bar 0}$-submodule of $\v$ generated by $f\cdot w_j,$ then  $\v_{\bar 1}=P+ \sum_{j=1}^n U^j$ in which   $P:=\{0\}$ if $\v_{\bar 1}$ has no one-dimensional irreducible $\fg_{\bar 0}$-submodule and otherwise, we take it to be the summation of all one-dimensional irreducible $\fg_{\bar 0}$-submodules of $\v_{\bar 1}:$
Take $U:=P+ \sum_{j=1}^n U^j.$ Since $\v_{\bar 1}$ is a completely reducible $\fg_{\bar 0}$-module, there is a $\fg_{\bar 0}$-submodule $K$ of $\v_{\bar 1}$ such that $\v_{\bar 1}=U\op K.$ If $K\neq \{0\},$ we pick an irreducible $\fg_{\bar 0}$-submodule $S$ of $K$ and suppose $u$ is an eigenvector for the largest  eigenvalue $\lam$ of the action $\frac{1}{2}h$ on $S.$ Since $S\cap U\sub K\cap U=\{0\},$ $S$ is not one-dimensional, so $\lam$ is positive  and $2f\cdot (f\cdot u)=[f,f]\cdot u\neq 0.$ But $f\cdot u\in \v_{\bar 0}=\sum_{j=1}^n\sum_{k=0}^{\lam_j}\bbbf f^{2k}\cdot w_j,$ so $f\cdot(f\cdot u)\in \sum_{j=1}^n\sum_{k=0}^{\lam_j}\bbbf f^{2k}\cdot f\cdot w_j\in U.$ This means that $0\neq f\cdot(f\cdot u)\in S\cap U\sub K\cap U=\{0\}, $ a contradiction. Therefore, $K=\{0\}$ and we are done.

\textbf{Step 5}. $\v$ is  a summation of irreducible  $\fg$-modules: Consider $U^j$ and $P$ as in Step 4.
Fix a basis $\{v_1,\ldots,v_m\}$ of $P$ if $P$ is not zero  and set $x_i:=f\cdot(e\cdot v_i)-2v_i$ for $i\in\{1,\ldots,m\}.$ Then for $i\in\{1,\ldots,m\},$ $\bbbf x_i$ is a trivial  one-dimensional  $\fg$-submodule of $\v$ and  as $e\cdot v_i\in \v_{\bar0},$ (\ref{final})  implies that $v_i=\frac{1}{2}(f\cdot(e\cdot v_i)-x_i)\in \sum_{j=1}^n U^j+\bbbf x_i.$ Moreover, by Lemma \ref{gen}, for $j\in\{1,\ldots,n\},$ $U^j=T^j+S^j $ and  by Steps 2,3,4, we have $$\v=\sum_{j=1}^n\sum_{k=0}^{2\lam_j+2}\bbbf f^{k}\cdot (e\cdot w_j)+\sum_{j=1}^n\sum_{k=0}^{2\lam_j}\bbbf f^{k}\cdot(f\cdot (e\cdot w_j)-(\lam_j+2)w_j)+\sum_{i=1}^m\bbbf x_i$$ in which the last part is disappeared if $P=\{0\}.$ This together with Step 1 and Corollary \ref{cor1} completes the proof.
\qed

\section{Extended Affine Lie Superalgebras}

 We call a triple   $(\LL,\hh,\fm),$ consisting of a nonzero Lie superalgebra $\LL=\LL_{\bar 0}\op\LL_{\bar1},$ a nontrivial subalgebra $\hh$ of $\LL_{\bar 0}$ and a nondegenerate  invariant even supersymmetric bilinear form $\fm$ on $\LL,$ a {\it super-toral triple} if
\begin{itemize}
\item
{\rm $\LL$ has a  weight space decomposition $\LL=\op_{\a\in\hh^*}\LL^\a$ with respect to $\hh$ via the adjoint representation. We note that in this case $\hh$ is abelian; also as $\LL_{\bar 0}$ as well as $\LL_{\bar 1}$ are $\hh$-submodules of $\LL,$ we have $\LL_{\bar 0}=\op_{\a\in \hh^*}\LL_{\bar 0}^\a$ and $\LL_{\bar 1}=\op_{\a\in\hh^*}\LL_{\bar 1}^\a$  with $\LL_{\bar i}^\a:=\LL_{\bar i}\cap \LL^\a,$ $ i=0,1$ \cite[Pro. 2.1.1]{MP}},

 \item {\rm the restriction of the form  $\fm$ to  $\hh$ is nondegenerate.}
\end{itemize}
 We call  $R:=\{\a\in \hh^*\mid \LL^\a\neq\{0\}\},$ the {\it root system} of $\LL$ (with respect to $\hh$). Each element of $R$ is called  a {\it root.} We refer to elements of $R_0:=\{\a\in \hh^*\mid \LL_{\bar 0}^\a\neq\{0\}\}$ (resp. $R_1:=\{\a\in \hh^*\mid \LL_{\bar1}^\a\neq\{0\}\}$) as {\it even roots} (resp. {\it odd roots}). We note that $R=R_0\cup R_1.$
Suppose that $(\LL,\hh,\fm)$ is a super-toral triple and  $\mathfrak{p}:\hh\longrightarrow \hh^*$ is the function  mapping   $h\in\hh$ to $(h,\cdot).$ Since the form is nondegenerate on $\hh,$ this map is one to one. So for each element $\a$ of  the image $\hh^\mathfrak{p}$ of $\hh$ under $\mathfrak{p},$ there is a unique $t_\a\in\hh$ representing $\a$ through the form $\fm.$ Now we can transfer the form on $\hh$ to a form on $\hh^\mathfrak{p},$ denoted again by $\fm,$  and defined by $$(\a,\b):=(t_\a,t_\b)\;\;\;(\a,\b\in \hh^\fp).$$

\begin{lem}
\label{symm} Suppose that $(\LL,\hh,\fm)$ is a  super-toral triple with corresponding root system $R=R_0\cup R_1.$ Then we have the following:

(i) For $\a,\b\in \hh^*,$ $[\LL^\a,\LL^\b]\sub\LL^{\a+\b}.$ Also for $i=0,1$ and $\a,\b\in R_i,$ we have $(\LL_{\bar i}^\a,\LL_{\bar i}^\b)=\{0\}$ unless $\a+\b=0;$ in particular, $R_0=-R_0$ and $R_1=-R_1.$

(ii) Suppose that  $\a\in \hh^\fp$ and $x_{\pm\a}\in\LL^{\pm\a}$ with $[x_\a,x_{-\a}]\in\hh,$ then we have $[x_\a,x_{-\a}]=(x_\a,x_{-\a})t_\a.$

(iii) Suppose that  $\a\in R_i\setminus\{0\}$ $(i\in\{0,1\}),$  $x_\a\in\LL_{\bar i}^\a$ and $x_{-\a}\in\LL_{\bar i}^{-\a}$ with  $[x_\a,x_{-\a}]\in\hh\setminus\{0\},$ then  we have $(x_\a,x_{-\a})\neq0$ and that  $\a\in\hh^\fp.$
\end{lem}
\pf $(i)$ It is easy to see.

$(ii)$ For $h\in\hh,$ we have \begin{equation}\label{p}(h,[x_\a,x_{-\a}])=([h,x_\a],x_{-\a})=\a(h)(x_\a,x_{-\a}).\end{equation}  Therefore we have  $$(h,[x_\a,x_{-\a}])=\a(h)(x_\a,x_{-\a})=(t_\a(x_\a,x_{-\a}),h).$$ This together with  the fact that the form on $\hh$ is symmetric and nondegenerate completes the proof.

$(iii)$ Suppose to the contrary that $(x_\a,x_{-\a})=0,$ then (\ref{p}) implies that for all $h\in\hh,$ $(h,[x_\a,x_{-\a}])=0$  but the form on $\hh$ is nondegenerate, so $[x_\a,x_{-\a}]=0,$  a contradiction. Again  using (\ref{p}), we get that $(h,\frac{1}{(x_\a,x_{-\a})}[x_\a,x_{-\a}])=\a(h)$   for all $h\in\hh$ and so $\a=\mathfrak{p}(\frac{1}{(x_\a,x_{-\a})}[x_\a,x_{-\a}])\in\hh^\fp.$
 \qed
\begin{definition}
{\rm A super-toral triple  $(\LL=\LL_{\bar 0}\op\LL_{\bar1},\hh,\fm)$ (or $\LL$ if there is no confusion),  with root system $R=R_0\cup R_1,$ is called an {\it extended affine Lie superalgebra} if}
\begin{itemize}{\rm
\item{(1)}  for $\a\in R_i\setminus\{0\}$ ($i\in\{0,1\}$), there are $x_\a\in\LL_{\bar i}^\a$ and $x_{-\a}\in\LL_{\bar i}^{-\a}$ such that  $0\neq[x_\a,x_{-\a}]\in\hh,$
    \item{(2)}  for $\a\in R$ with $(\a,\a)\neq 0$ and $x\in \LL^\a,$ $ad_x:\LL\longrightarrow\LL,$ mapping $y\in\LL$ to $[x,y],$ is a locally nilpotent linear transformation.}
\end{itemize}

{\rm The extended affine Lie superalgebra $(\LL,\hh,\fm)$ is called an {\it invariant affine reflection algebra} \cite{N1} if $\LL_{\bar1}=\{0\}$ and it is called a {\it locally extended affine Lie algebra} \cite{MY} if $\LL_{\bar1}=\{0\}$ and
$\LL^0=\hh.$ Finally  a locally extended affine Lie algebra $(\LL,\hh,\fm)$ is called an  {\it extended
affine Lie algebra} \cite{AABGP} if $\LL^0=\hh$ is a finite dimension subalgebra of $\LL.$}
\end{definition} We immediately  have  the following lemma:
\begin{prop}\label{aff-ref}
If $(\LL,\hh,\fm)$ is an extended affine Lie superalgebra, then the triple  $(\LL_{\bar 0},\hh,\fm|_{\LL_{\bar 0}\times\LL_{\bar 0}})$ is an invariant affine reflection algebra.
\end{prop}

\begin{Example}
{\rm Finite dimensional  basic classical simple Lie superalgebras \cite{K1} and  affine Lie superalgebras \cite{van-de} are examples of extended affine Lie superalgebras.}\hfill $\diamondsuit$
\end{Example}

In the sequel,  we shall  prove that the  root system of an extended affine Lie superalgebra $(\LL,\hh,\fm)$ is an extended affine root supersystem  in the following sense:

\begin{definition}\label{iarr}
{\rm Suppose that $A$ is a nontrivial additive abelian group, $\fm:A\times A\longrightarrow \bbbf$ is  a symmetric form  and $R$ is a subset of $A.$ Set
$$\begin{array}{l}
R^0:=R\cap A^0,\;\;\;\;\;
\rcross:=R\setminus R^0,\\\\
\rcross_{re}:=\{\a\in R\mid (\a,\a)\neq0\},\;\;\;\rre:=\rcross_{re}\cup\{0\},\\\\
\rcross_{ns}:=\{\a\in R\setminus R^0\mid (\a,\a)=0 \},\;\;\; \rim:=\rcross_{ns}\cup\{0\}.
\end{array}$$
We say $(A,\fm,R)$ is an {\it extended affine root supersystem}  if the following hold:
$$\begin{array}{ll}
(S1)& \hbox{$0\in R,$ and $\la R\ra= A,$}\\\\
(S2)& \hbox{$R=-R,$}\\\\
(S3)&\hbox{for $\a\in \rre^\times$ and $\b\in R,$ $2(\a,\b)/(\a,\a)\in\bbbz,$}\\\\
(S4)&\parbox{4.5in}{ {\it root string property} holds in $R$ in the sense that for $\a\in \rre^\times$ and $\b\in R,$  there are nonnegative  integers  $p,q$  with $2(\b,\a)/(\a,\a)=p-q$ such that \begin{center}$\{\b+k\a\mid k\in\bbbz\}\cap R=\{\b-p\a,\ldots,\b+q\a\},$\end{center}
} \\\\
(S5)&\parbox{4.5in}{for $\a\in \rim$ and $\b\in R$ with $(\a,\b)\neq 0,$
$\{\b-\a,\b+\a\}\cap R\neq \emptyset.$ }
\end{array}$$
If there is no confusion, for the sake of simplicity, we say   {\it $R$ is an extended affine root supersystem in $A.$}}
\end{definition}
%
%

\begin{lem}\label{sl2}
Suppose that   $(\LL,\hh,\fm)$ is an extended affine Lie superalgebra  with root system $R=R_0\cup R_1.$ Suppose that  $\a\in R_i$ $(i\in\{0,1\})$  with $(\a,\a)\neq 0.$ Recall that $t_\a$ is the unique element of $\hh$ representing $\a$ through the form $\fm$ restricted to $\hh$ and set $h_\a:=2t_\a/(\a,\a).$ Then  there are $y_{\pm\a}\in\LL_{\bar i}^{\pm \a}$ such that $(y_\a,y_{-\a},h_\a)$ is an $\mathfrak{sl}_2$-super triple for the subsuperalgebra $\gg(\a)$  generated by $\{y_\a,y_{-\a},h_\a\};$ in particular, if $\a\in R_1\cap \rre^\times,$ then $2\a\in R_0.$
\end{lem}
\pf Suppose that $i\in\{0,1\}$ and $\a\in R_i$ with $(\a,\a)\neq0.$ Fix $x_{\pm\a}\in\LL_{\bar i}^{\pm\a}$ with $0\neq[x_\a,x_{-\a}]\in\hh.$
Considering Lemma \ref{symm} and setting  $e_\a:=x_\a$ and $e_{-\a}:=x_{-\a}/(x_\a,x_{-\a}),$ we have  $[e_\a,e_{-\a}]=t_\a.$ Now  we get that   $(y_\a:=2e_\a/(t_\a,t_\a),y_{-\a}:=e_{-\a},h_\a)$ is an $\mathfrak{sl}_2$-super triple for the subsuperalgebra $\gg(\a).$  Next suppose $\a\in R_1\cap\rre^\times.$ Using Lemma \ref{iso}, we get that   $0\neq[y_\a,y_\a]\sub\LL^{2\a}_{\bar 0}$ and so $2\a\in R_0.$
\qed

\begin{lem}\label{final3}
 If  $i,j\in\{0,1\},$ $\a\in R_i$ and $\b\in R_j$ with $(\a,\b)\neq0,$ then either $\b-\a\in R$ or $\b+\a\in R.$
\end{lem}
\pf Fix $0\neq z\in\LL_{\bar j}^\b$ and  $x\in\LL_{\bar i}^\a,y\in\LL_{\bar i}^{-\a}$ with $[x,y]\in \hh\setminus\{0\}.$ Using Lemma \ref{symm}, we have $[x,y]=(x,y)t_\a.$
Therefore we have
\begin{eqnarray*}
0\neq(x,y)(\a,\b)z=(x,y)[t_\a,z]=[[x,y],z]=[x,[y,z]]-(-1)^{|x||y|}[y,[x,z]].
\end{eqnarray*}
This in turn implies that either $[y,z]\neq 0$
or $[x,z]\neq 0.$ Therefore either $\b-\a\in R$ or $\b+\a\in R.$
 \qed

\begin{prop}\label{no hole}Suppose that   $(\LL,\hh,\fm)$ is an extended affine Lie superalgebra  with root system $R=R_0\cup R_1.$
For $\a,\b\in R$  with $(\a,\a)\neq0,$ we have the following:

(i) $\frac{2(\b,\a)}{(\a,\a)}\in\bbbz,$ in particular if $k\in\bbbf$ and $k\a\in R,$ then $k\in\{0,\pm1,\pm2,\pm1/2\}.$

(ii) $r_\a(\b):=\b-\frac{2(\a,\b)}{(\a,\a)}\a\in R.$

(iii) There are nonnegative  integers $p,q$ such that $p-q=2(\b,\a)/(\a,\a)$ and $\{k\in\bbbz\mid \b+k\a\in R\}=\{-p,\ldots,q\}.$
\end{prop}
\pf
Suppose that $\a,\b\in R$ with $(\a,\a)\neq0.$ Assume  $\a\in R_i$ for some $i\in\{0,1\}.$ Using Lemma \ref{sl2},  there are $y_{\a}\in\LL_{\bar i}^{\a}, y_{-\a}\in\LL_{\bar i}^{- \a}$ such that $[y_\a,y_{-\a}]=h_\a=\frac{2t_\a}{(\a,\a)}$ and $(y_\a,y_{-\a},h_\a)$ is an $\mathfrak{sl}_2$-super triple for the subsuperalgebra $\gg(\a)$ of $\LL$ generated by $\{y_\a,y_{-\a},h_\a\}.$  Consider $\LL$ as a  $\gg(\a)$-module via the adjoint representation, then  $\m:=\sum_{k\in\bbbz}\LL^{\b+k\a}$  is a $\gg(\a)$-submodule of $\LL.$  For $k\in\bbbz$ and $x\in\LL^{\b+k\a},$ set  $$\m(x):=\hbox{span}_\bbbf\{\hbox{ad}_{y_{-\a}}^n\hbox{ad}_{y_\a}^mx\mid m,n\in\bbbz^{\geq 0}\}.$$ We claim that $\m(x)$ is  the $\fg(\a)$-submodule of $\m$ generated by $x.$ Indeed, as $\gg(\a)$ is generated by $\{y_\a,y_{-\a},h_\a\},$ it is enough to show that $\m(x)$ is invariant  under  $\hbox{ad}_{h_\a},\hbox{ad}_{y_\a},\hbox{ad}_{y_{-\a}}.$ By Lemma \ref{symm}($i$), for $m,n\in\bbbz^{\geq0},$ we have $\hbox{ad}_{y_{-\a}}^n\hbox{ad}_{y_\a}^mx\in \LL^{\b+k\a+m\a-n\a},$ so $\m(x)$ is invariant under the action of $h_\a.$ Also it is trivial that $\m(x)$ is invariant under the action of $y_{-\a}.$ We finally show that it is invariant under  $\hbox{ad}_{y_\a}.$ We use  induction on $n$ to prove that $[y_\a,\hbox{ad}_{y_{-\a}}^n\hbox{ad}_{y_\a}^mx]\in\m(x)$ for $n,m\in\bbbz^{\geq0}.$
If $n=0,$ there is nothing to prove, so we assume $n\in\bbbz^{\geq1}$ and that $[y_\a,\hbox{ad}_{y_{-\a}}^{n-1}\hbox{ad}_{y_\a}^mx]\in\m(x)$ for all $m\in\bbbz^{\geq0}.$  Now for $m\in\bbbz^{\geq0},$  we have
\begin{eqnarray*}
[y_\a,\hbox{ad}_{y_{-\a}}^n\hbox{ad}_{y_\a}^mx]&=&[y_\a,[y_{-\a},\hbox{ad}_{y_{-\a}}^{n-1}\hbox{ad}_{y_\a}^mx]]\\
&=&(-1)^{|y_\a|}[y_{-\a},[y_{\a},\hbox{ad}_{y_{-\a}}^{n-1}\hbox{ad}_{y_\a}^mx]]+[h_\a,\hbox{ad}_{y_{-\a}}^{n-1}\hbox{ad}_{y_\a}^mx].
\end{eqnarray*}
This together with the induction hypothesis and the fact that $\m(x)$ is invariant under  $\hbox{ad}_{h_\a}$ and $\hbox{ad}_{y_{-\a}},$ completes the induction process. Now we are ready to prove the proposition. Keep the notations as above.

$(i)$ Since $\hbox{ad}_{y_{-\a}}$ and $\hbox{ad}_{y_\a}$ are locally nilpotent linear transformations, for $x\in\LL^{\b+k\a}$ $(k\in\bbbz),$ $\m(x)$  is  finite dimensional, so $\m$ is a  summation of the finite dimensional  $\gg(\a)$-submodules $\m(x)$ $(x\in\LL^{\b+k\a};\; k\in\bbbz).$ We know that
\begin{equation}\label{eigen}
\parbox{3.8in}{
$h_\a$ acts  diagonally on $\m$ with the set of eigenvalues  $\{\b(h_\a)+2k\mid k\in\bbbz,\;\b+k\a\in R\}.$ Moreover, this set of eigenvalues  is the union of the set of eigenvalues of the action of ${h_\a}$ on the finite dimensional $\gg(\a)$-submodules  $\m(x)$ ($x\in\LL^{\b+k\a};\;k\in\bbbz$) of $\m.$}
\end{equation}
Since $\LL^\b\neq \{0\},$ each nonzero element of $\LL^\b$ is an eigenvector of $\hbox{ad}_{h_\a}$ restricted to $\m$ corresponding to the eigenvalue $\b(h_\a).$ Therefore $\b(h_\a)$  is an eigenvalue of $\hbox{ad}_{h_\a}$ restricted to a finite dimensional $\gg(\a)$-submodule $\m(x)$ for some $x\in\LL^{\b+k\a}$ $(k\in\bbbz)$ and so using  $\frak{sl}_2$-module theory together with  Lemma \ref{int eigen}, we get that   $2(\b,\a)/(\a,\a)=\b(h_\a)\in\bbbz.$ This completes the proof.

$(ii)$ As in the previous case, $\b(h_\a)$ is an eigenvalue of  $\hbox{ad}_{h_\a}$ restricted to    a finite dimensional $\gg(\a)$-submodule $\m(x)$ of $\m,$ for some $x\in\LL^{\b+k\a}$ ($k\in\bbbz$).  From Lemma \ref{int eigen} and $\frak{sl}_2$-module theory, we know that $-\b(h_\a)$ is also an eigenvalue for $\hbox{ad}_{h_\a}$ on $\m(x).$ So there is an integer  $k$ such that $\b+k\a\in R$  and $-\b(h_\a)=\b(h_\a)+2k.$ This implies that $k=-\b(h_\a).$ In particular, $\b-2\frac{(\b,\a)}{(\a,\a)}\a=\b-\b(h_\a)\a=\b+k\a\in R.$

$(iii)$
We first prove that $\{k\in\bbbz\mid \b+k\a\in R\}$ is an interval.
To this end, we take  $r,s\in\bbbz$  with $\b+r\a,\b+s\a\in R$ and show that   $\b+t\a\in R$ for all $t$ between $r,s.$  Without loss of generality, we may assume $|\b(h_\a)+2r|\geq|\b(h_\a)+2s|.$ Suppose that   $t$ is an integer between $r,s.$
Since $\b+t\a\in R$ if and only if $-\b-t\a\in R,$ we  simultaneously replace $\b$ with $-\b$ and $(r,s)$ with $(-r,-s)$ if it is necessary and  assume $\b(h_\a)+2r\geq0.$ So $-\b(h_\a)-2r\leq \b(h_\a)+2s\leq \b(h_\a)+2r.$ But using (\ref{eigen}), we get that $\b(h_\a)+2r$ is an eigenvalue of the action of $h_\a$ on $\m(x)$ for some $x\in\LL^{\b+k\a}$ $(k\in\bbbz).$ So by  Lemmas \ref{int eigen} and \ref{com red},  $\b(h_\a)+2t$ is an eigenvalue for the action of $h_\a$ on $\m(x).$  Therefore, we have  $\b+t\a\in R$ by (\ref{eigen}).
We next show that this interval is bounded.
For $k\in\bbbz,$ we have  $(\b+k\a,\b+k\a)=(\b,\b)+2k(\b,\a)+k^2(\a,\a).$  So there are at  most two integer numbers such that $\b+k\a\not\in \rre^\times.$ Now to the contrary assume that  the mentioned interval is not bounded. Without loss of generality, we may assume  there is a positive integer $k_0$ such that for $k\in\bbbz^{\geq k_0},$ $\b+k\a\in \rre^\times$ and $(\a,\b+k\a)\neq 0.$ For $k\in\bbbz^{\geq k_0},$ we have
\begin{eqnarray*}
\frac{\frac{2(\a,\b)}{(\a,\a)}+2k}{\frac{(\b,\b)}{(\a,\a)}+\frac{2k(\a,\b)}{(\a,\a)}+k^2}&=&\frac{2(\a,\b)+2k(\a,\a)}{(\b,\b)+2k(\a,\b)+k^2(\a,\a)}\\
&=&\frac{2(\a,\b+k\a)}{(\b+k\a,\b+k\a)}\in\bbbz.\end{eqnarray*}
Now as $k,\frac{2(\a,\b)}{(\a,\a)}\in\bbbz,$ we get  that $\frac{(\b,\b)}{(\a,\a)}\in\bbbq,$ so $\lim_{k\rightarrow \infty}\frac{\frac{2(\a,\b)}{(\a,\a)}+2k}{\frac{(\b,\b)}{(\a,\a)}+\frac{2k(\a,\b)}{(\a,\a)}+k^2}=0.$ This  is a contradiction as it is a  sequence of nonzero integer numbers. Therefore we have a bounded interval. Suppose $p,q$ are the largest nonnegative integers with $\b-p\a,\b+q\a\in R.$ Since $r_\a(\b-p\a)=\b+q\a,$ we get that $p-q=\frac{2(\b,\a)}{(\a,\a)}.$ This completes the proof.
\qed

\begin{cor}
\label{cor2}
Suppose that $(\LL,\fm,\hh)$ is an extended affine Lie superalgebra with root system $R,$ then $R$ is an extended affine root supersystem in its $\bbbz$-span.
\end{cor}
\begin{proof}
It is immediate using Lemma \ref{final3} together with Proposition \ref{no hole}.
\end{proof}

\begin{prop}
\label{symm2}
Suppose that $(\LL,\fm,\hh)$ is an extended affine Lie superalgebra with root system $R.$

(i) For $\a\in \rre^\times,$ we have $2\a\not \in R_1;$ also if  $\a\in \rre$ and $2\a\not\in R,$ we have $\a\in  R_0.$

(ii) If $\a\in R_0$ with $(\a,\a)=0,$ then $(\a, R_0)=\{0\};$ moreover,  $R_0\cap \rim=\{0\}.$

(iii) If $\LL^0\sub \LL_{\bar 0},$ then we have    $ R^\times\cap R_0\cap R_1=\emptyset.$
\end{prop}
\pf
$(i)$  We know from Proposition \ref{no hole}($i$) that $4\a\not\in R.$ Now the result is immediate using Lemma \ref{sl2}($i$).

$(ii)$  Although  using  a modified argument as in \cite[Pro. 3.4]{MY} and \cite[Pro. I.2.1]{AABGP}, one can get the first assertion, for the convenience of readers, we give its proof. To the contrary, suppose $\a,\b\in R_0$ with $(\a,\a)= 0$ and $(\a,\b)\neq 0.$
If there are infinitely many  consecutive integers $n$ with $\b+n\a\in R_0\cap \rre^\times,$  for such integer numbers  $n,$ we have $\frac{2(\a,\b+n\a)}{(\b+n\a,\b+n\a)}=\frac{2(\a,\b)}{(\b,\b)+2n(\a,\b)}\in\bbbz.$ Therefore, we get $(\b,\b)\neq 0$ and $\frac{\frac{2(\a,\b)}{(\b,\b)}}{1+2n\frac{(\a,\b)}{(\b,\b)}}\in\bbbz$ which is absurd as it is a sequens of nonzero integer  numbers converging  to $0.$ Therefore, there is  $p\in \bbbz$ with $\gamma:=\b+p\a\in R_0$ and $\gamma-\a\not \in R_0.$
Fix $0\neq x\in\LL_{\bar 0}^\gamma$ and $y_{\pm\a}\in \LL_{\bar 0}^{\pm\a}$ with $[y_{-\a},y_\a]=t_\a.$ Setting $x_0:=x, x_n:=(ad_{y_\a})^nx,$ for $n\in\bbbz^{\geq 1},$ we have $ad_{y_{-\a}}(x_n)=n(\a,\gamma)x_{n-1}.$ So $x_n\neq0$ for all  $n\in\bbbz^{\geq 0}.$ This means that for consecutive integer numbers $k_n:=p+n$ $(n\in\bbbz^{\geq0}),$  $\b+k_n\a\in R_0,$  a contradiction.

For the last assertion, suppose $0\neq \a\in R_0\cap \rim.$ Since $\a\not\in R^0,$ there is  $\b\in R$ with $(\a,\b)\neq 0.$ If $\b\in \rim,$   there is $r\in\{\pm 1\}$  with $\a+r\b\in \rre;$ see Lemma \ref{final3}. Since $(\a+r\b,\a)\neq0,$ we always may assume there is a real root $\gamma$ with $(\a,\gamma)\neq0.$ But by the first assertion, $\gamma\in R_1\cap \rre^\times$ and so $2\gamma\in R_0$ by Lemma \ref{sl2}. This  implies that $(\a,2\gamma)=0$ which is a contradiction. This completes the proof.

$(iii)$ To the contrary, suppose that $\a\in R^\times\cap R_0\cap R_1.$ By part ($ii$), we get that $\a\in\rre^\times.$ Consider Lemma \ref{symm} and fix $x_\a\in\LL_{\bar 0}^\a,y_\a\in \LL_{\bar 0}^{-\a}$ such that $[x_\a,y_\a]=t_\a;$  also fix $e_\a\in\LL_{\bar1}^\a,$ $f_\a\in\LL_{\bar1}^{-\a}$ with $[e_\a,f_\a]=t_\a.$ We have  $[y_\a,e_\a]\in\LL_{\bar1}^0=\{0\}$ and by part ($i$), $[x_\a,e_\a]\in\LL_{\bar1}^{2\a}=\{0\}$,  so we get
\begin{eqnarray*}
0\neq (\a,\a)e_\a=[t_\a,e_\a]=[[x_\a,y_{\a}],e_\a]=[x_\a,[y_\a,e_\a]]-[y_\a,[x_\a,e_\a]]=0
\end{eqnarray*}
which   is a contradiction.
\qed

\begin{Example}{\rm
Suppose that  $\Lam$ is a torsion free additive abelian group and $\gg$ is  a locally finite  basic classical simple Lie superalgebra i.e., a direct union of finite dimensional basic classical simple Lie superalgebras (see \cite{you7} for details) with a Cartan subalgebra  $\hh$ and a fixed nondegenerate  invariant even bilinear form $f\fm.$ One knows  $(\gg,\hh,f\fm)$ is an extended affine Lie superalgebra with $\gg^0=\hh.$
 Suppose that $\theta:\Lam\times\Lam\longrightarrow \bbbf\setminus\{0\}$ is a commutative 2-cocycle, that is, $\theta$ satisfies the following properties:
$$\theta(\zeta,\xi)=\theta(\xi,\zeta)\andd \theta(\zeta,\xi)\theta(\zeta+\xi,\eta)=\theta(\xi,\eta)\theta(\zeta,\xi+\eta)$$
for all $\zeta,\xi,\eta\in \Lam$.
Suppose that $\theta(0,0)=1$ and note that this in turn  implies that $\theta(0,\lam)=1$  for all $\lam\in\Lam.$ Consider the $\bbbf$-vector space $\aa:=\sum_{\lam\in\Lam}\bbbf t^\lam$ with a basis $\{t^\lam\mid \lam\in\Lam\}.$ Now  $\aa$ together with the product defined by $$t^\zeta\cdot t^\xi:=\theta(\zeta,\xi)t^{\zeta+\xi}\;\;\;(\xi,\zeta\in\Lam)$$ is a $\Lam$-graded  unital commutative associative algebra with $\aa^\lam:=\bbbf t^\lam$ (for $\lam\in\Lam$). We refer to $\aa$ as the {\it commutative  associative torus} corresponding to $(\Lam,\theta).$
Set $$\hat \gg:=\gg\ot\aa$$ and define $$|x\ot a|:=|x|;\;\; x\in\gg,a\in\aa.$$ Then $\hat\gg$ together with
$$[x\ot a,y\ot b]_{_{\hat \gg}}:=[x,y]\ot ab$$  for $x,y\in\gg$ and $a,b\in\aa,$ is a Lie superalgebra.
Now define $$(x\ot t^\lam,y\ot t^\mu):=\theta(\lam,\mu)\d_{\lam+\mu,0}f(x,y)$$ for $x,y\in \gg$ and $\lam,\mu\in\Lam.$
This defines  a nondegenerate  invariant even supersymmetric bilinear  form on $\hat\gg.$
Next take $\v:=\bbbf\ot_\bbbz\Lam.$ Identify $\Lam$ with a subset of $\v$ and  fix a basis $B:=\{\lam_i\mid i\in I\}\sub \Lam$ of $\v.$ Suppose that $\{d_i\mid i\in I\}$ is the dual basis of $B$ and  $\v^\dag$ is the restricted dual of $\v$ with respect to this basis.
Each $d\in \v^\dag$ can be also  considered as a derivation on  $\hat \gg$ mapping $x\ot t^\lam$ to $d(\lambda) x\ot t^\lam$  for   $x\in  \gg$ and $\lambda \in \Lambda;$ indeed, for $a,b\in \gg$ and $\lam,\mu\in\Lam,$ we have \begin{eqnarray}
d([a\ot t^\lam,b\ot  t^\mu]_{_{\hat\gg}})&=&d(\lam+\mu)[a,b]\ot t^\lam t^\mu\nonumber\\
&=&d(\lam)[a,b]\ot t^\lam t^\mu+d(\mu)[a,b]\ot t^\lam t^\mu\label{e2}\nonumber\\
&=&[d(a\ot t^\lam),b\ot t^\mu]_{_{\hat\gg}}+[a\ot t^\lam,d(b\ot t^\mu)]_{_{\hat\gg}}.\nonumber
\end{eqnarray}

Also  for $a,b\in \gg,$ $d,d'\in \v^\dag$  and $\lam,\mu\in \Lam, $  we have
\begin{eqnarray}(d(a\ot t^\lam),b\ot t^\mu)=d(\lam)(a\ot t^\lam,b\ot t^\mu)
&=&d(\lam)\d_{\lam,-\mu}\theta(\lam,\mu)f(a,b)\nonumber\\
&=&-d(\mu)\d_{\lam,-\mu}\theta(\lam,\mu)f(a,b)\nonumber\\
&=&-d(\mu)(a\ot t^\lam,b\ot t^\mu)\nonumber\\
&=&-(a\ot t^\lam,d(b\ot t^\mu))\nonumber
\end{eqnarray}
and
\begin{eqnarray}
(dd'(a\ot t^\lam),b\ot t^\mu)=d(\lam)d'(\lam)(a\ot t^\lam,b\ot t^\mu)&=&d(\lam)d'(\lam)\d_{\lam+\mu,0}\theta(\lam,\mu)f(a,b)\nonumber\\
&=&-d(\lam)d'(\mu)\d_{\lam+\mu,0}\theta(\lam,\mu)f(a,b)\nonumber\\
&=&-d(\lam)d'(\mu)(a\ot t^\lam,b\ot t^\mu)\nonumber\\
&=&-(d(\lam)a\ot t^\lam,d'(\mu)b\ot t^\mu)\nonumber\\
&=&-(d(a\ot t^\lam),d'(b\ot t^\mu)).\nonumber
\end{eqnarray}

Therefore, we have
{\small\begin{equation}\label{e1}
(d(x),y)=-(x,d(y))\andd (dd'(x),y)=-(d(x),d'(y))\;\;\;\;\; (x,y\in \hat\gg,d,d'\in\v^\dag).
\end{equation}}
 Set
$$\fl:=\hat \gg\op\v\op\v^\dag=(\gg\ot \aa)\op\v\op\v^\dag$$ and define

\begin{equation}\label{bracketloop2}\begin{array}{l}
  \;[d,x]=-[x,d]=d(x),  \quad d\in\v^\dag,x\in  \hat \gg,\\
 \;[\v,\fl]=[\fl,\v]=\{0\},\\
 \;[\v^\dag,\v^\dag]=\{0\},\\
\;[x,y]=[x,y]_{_{\hat\gg}}+\sum_{i\in I}(d_i(x),y)\lam_i, \quad
x,y\in \hat\gg. \end{array}\end{equation}

\begin{lem}
Set $\fl_{\bar 0}:=\hat\gg_{\bar 0}\op\v\op\v^\dag$ and $\fl_{\bar 1}:=\hat\gg_{\bar 1},$ then  $\fl=\fl_{\bar 0}\op\fl_{\bar 1}$ together with the Lie bracket as in (\ref{bracketloop2}) is a Lie superalgebra.
\end{lem}
\begin{proof} Since the form on $\gg$ is supersymmetric, (\ref{e1}) implies that the Lie bracket defined in  $(\ref{bracketloop2})$ is anti-supercommutative. So we just  need to verify the Jacobi superidentity.
We recall  that  the form on $\hat\gg$ is invariant and supersymmetric and that  $d\in \v^\dag$ acts as  a derivation on $\hat \gg.$  Take  $x,y,z\in \hat \gg$ and $d,d'\in\v^\dag.$ Then  we have
{\small
\begin{eqnarray*}
(d([x,y]_{\hat \gg}),z)
&\stackrel{(\ref{e1})}{=}&-([x,y]_{\hat \gg},d(z))\\
&=&-(x,[y,d(z)]_{\hat \gg})\\
&=&(-1)^{|y||z|}(x,[d(z),y]_{\hat \gg})\\
&=&(-1)^{|y||z|}(x,d([z,y]_{\hat \gg}))-(-1)^{|y||z|}(x,[z,d(y)]_{\hat \gg})\\
&=&-(x,d([y,z]_{\hat \gg}]))-(-1)^{|y||z|}([x,z]_{\hat \gg},d(y))\\
&=&-(x,d([y,z]_{\hat \gg}))+(-1)^{|y||z|+|z||x|}([z,x]_{\hat \gg},d(y))\\
&\stackrel{(\ref{e1})}{=}&-(-1)^{|x||z|}((-1)^{|x||y|}(d([y,z]_{\hat \gg}),x)+(-1)^{|y||z|}(d([z,x]_{\hat \gg}),y)).
\end{eqnarray*}}
Therefore we have
{\small \begin{eqnarray*}
&&(-1)^{|x||z|}[[x,y],z]+(-1)^{|z||y|}[[z,x],y]+(-1)^{|y||x|}[[y,z],x]\\
&=&(-1)^{|x||z|}[[x,y]_{\hat \gg},z]_{\hat \gg}+(-1)^{|x||z|}\sum_{i\in I}(d_i([x,y]_{\hat \gg}),z)\lam_i\\
&+&(-1)^{|z||y|}[[z,x]_{\hat \gg},y]_{\hat \gg}+(-1)^{|y||z|}\sum_{i\in I}(d_i([z,x]_{\hat \gg}),y)\lam_i\\
&+&(-1)^{|y||x|}[[y,z]_{\hat \gg},x]_{\hat \gg}+(-1)^{|x||y|}\sum_{i\in I}(d_i([y,z]_{\hat \gg}),x)\lam_i\\
&=&0.
\end{eqnarray*}}

Also we have
{\small
\begin{eqnarray*}
[[x,y],d]&=&[[x,y]_{_{\hat\gg}},d]\\
&=&-d([x,y]_{_{\hat\gg}}) \\
&=&-[d(x),y]_{_{\hat\gg}}-[x,d(y)]_{_{\hat\gg}}\\
&=&-[x,d(y)]_{_{\hat\gg}}+\sum_{i\in I}(dd_i(x),y)\lam_i-\sum_{i\in I}(dd_i(x),y)\lam_i-[d(x),y]_{_{\hat\gg}}\\
&=&-[x,d(y)]_{_{\hat\gg}}+\sum_{i\in I}(d_id(x),y)\lam_i-\sum_{i\in I}(dd_i(x),y)\lam_i-[d(x),y]_{_{\hat\gg}}\\
&\stackrel{(\ref{e1})}{=}&-[x,d(y)]_{_{\hat\gg}}-\sum_{i\in I}(d_i(x),d(y))\lam_i+\sum_{i\in I}(d(x),d_i(y))\lam_i-[d(x),y]_{_{\hat\gg}}\\
&=&-[x,d(y)]_{_{\hat\gg}}-\sum_{i\in I}(d_i(x),d(y))\lam_i+(-1)^{|x||y|}(\sum_{i\in I}(d_i(y),d(x))\lam_i+[y,d(x)]_{_{\hat\gg}})\\
&=&-[x,d(y)]+(-1)^{|x||y|}[y,d(x)]\\
&=&[x,[y,d]]-(-1)^{|x||y|}[y,[x,d]]
\end{eqnarray*}}
and

\begin{eqnarray*}
[[d,d'],x]=0=dd'(x)-dd'(x)=dd'(x)-d'd(x)=[d,[d',x]]-[d',[d,x]].
\end{eqnarray*}
Now the result  immediately follows.
\end{proof}
\begin{lem}
Extend the form on
$\hat\gg$
  to a supersymmetric bilinear form  $\fm_{_\fl}$ on $\fl$ by
\begin{equation}\label{formloop}
\begin{array}{l}
(\v,  \v)_{_\fl}=(\v^\dag,\v^\dag)_{_\fl}=(\v,\gg\ot\aa)_{_\fl}=(\v^\dag,\gg\ot\aa)_{_\fl}:=\{0\},\\
(v,d)_{_\fl}:=d(v),\quad  d\in \v^\dag, v\in \v.
\end{array}
\end{equation}
Then $\fm_{_\fl}$ is a nondegenerate invariant even supersymmetric bilinear form.
\end{lem}
\begin{proof}
It is trivial that this form is nondegenerate, even and supersymmetric, so we just prove that it is invariant.
We first consider the following easy table:
$${\small
\begin{tabular}{|l|l|l|}
\hline
$(x,y,z)\in$&$([x,y],z)\in $&$(x,[y,z])\in$\\
\hline
$(\hat\gg,\hat\gg,\v)$& $(\hat\gg+\v,\v)_{_\fl}=\{0\}$ &$ (\hat\gg,\{0\})_{_\fl}=\{0\}$\\
\hline
$(\v,\hat\gg,\hat\gg)$& $(\{0\},\hat\gg)_{_\fl}=\{0\}$ &$ (\v,\hat\gg+\v)_{_\fl}=\{0\}$\\
\hline
$(\hat\gg,\v, \v\cup\v^\dag\cup \hat\gg)$& $(\{0\},\fl)_{_\fl}=\{0\}$& $(\hat\gg,\{0\})_{_\fl}=\{0\}$\\
\hline
$(\hat \gg,\v^\dag,\v\cup\v^\dag)$&$(\hat\gg,\v\cup\v^\dag)_{_\fl}=\{0\}$& $(\hat\gg,\{0\})_{_\fl}=\{0\}$\\
\hline
$(\v\cup\v^\dag,\hat\gg,\v\cup\v^\dag)$& $(\hat\gg,\v\cup\v^\dag)_{_\fl}=\{0\}$&$(\v\cup\v^\dag,\hat\gg)_{_\fl}=\{0\}$\\
\hline
$ (\v\cup\v^\dag,\v\cup\v^\dag,\hat\gg)$ & $(\{0\},\hat\gg)_{_\fl}=\{0\}$ &$(\v\cup\v^\dag,\hat\gg)_{_\fl}=\{0\}$\\
\hline
$ (\v\cup\v^\dag,\v\cup\v^\dag,\v\cup\v^\dag)$ & $(\{0\},\v\cup\v^\dag)_{_\fl}=\{0\}$ &$(\v\cup\v^\dag,\{0\})_{_\fl}=\{0\}$\\
\hline
\end{tabular}}$$
Then we  note that if $x,y,z\in\hat\gg,$ we have  {\small$$([x,y],z)_{_\fl}=([x,y]_{_{\hat\gg}},z)_{_\fl}=([x,y]_{_{\hat\gg}},z)=(x,[y,z]_{_{\hat\gg}})=(x,[y,z]_{_{\hat\gg}})_{_\fl}=(x,[y,z])_{_\fl}$$}and for $x,y\in\hat\gg$ and $z=d_j\in\v^\dag$ $(j\in I),$ we get {\small \begin{eqnarray*}
([x,y],z)_{_\fl}=([x,y]_{\hat\gg}+\sum_{i\in I}(d_i(x),y)\lam_i,d_j)_{_\fl}=(d_j(x),y)
&\stackrel{(\ref{e1})}{=}&
-(x,d_j(y))\\
&=&-(x,[d_j,y])_{_\fl}\\
&=&(x,[y,z])_{_\fl}.
\end{eqnarray*}}
Considering the latter equality, as the form is supersymmetric,  for $y,z\in\hat\gg$ and $x=d_j\in\v^\dag$ $(j\in I),$ we have
 {\small \begin{eqnarray*}
([x,y],z)_{_\fl}=(-1)^{|y||z|}(z,[x,y])=-(-1)^{|y||z|}(z,[y,x])&=&-(-1)^{|y||z|}([z,y],x)\\
&=&-(-1)^{|y||z|}(x,[z,y])\\
&=&(x,[y,z]).
 \end{eqnarray*}}
Finally, for  $x,z\in\hat\gg$
and $y=d_j\in\v^\dag$ $(j\in I),$ one has {\small $$([x,y],z)_{_\fl}=-([d_j,x],z)_{_\fl}=-(d_j(x),z)\stackrel{(\ref{e1})}{=}(x,d_j(z))=(x,[d_j,z])=(x,[y,z])_{_\fl}.$$}
This completes the proof.
\end{proof}

Now set  $\fh:=(\hh\ot\bbbf)\op\v\op\v^\dag$ and take   $R$ to be  the root system of $\gg$ with respect to $\hh.$ We consider $\a\in R$ as an element of $\fh^*$ by $$\a({\v\op\v^\dag}):=\{0\} \andd \a(h\ot 1):=\a(h)\;\;\; (h\in\hh).$$ We also consider $\lam\in \v$ as an element of $\fh^*$ by $$\lam((\fh\ot\bbbf)\op\v):=\{0\}\andd \lam(d):=d(\lam)\;\;\; (d\in\v^\dag).$$ Then $\fl$ has  a weight space decomposition with respect to $\fh$  with the corresponding root system  $\mathfrak{R}=\{\a+\lam\mid \a\in R,\lam\in\Lam\};$  moreover, we have
$$\begin{array}{l}
\fl^0=\fh\andd \fl^{\a+\lam}=\gg^\a\ot \bbbf t^\lam \;\;\;(\a\in R,\lam\in\Lam \hbox{ with $\a+\lam\neq0$}).
\end{array}$$
Suppose $\lam\in \Lam$ and $\a\in R_i$ ($i\in\{0,1\}$) with $\a+\lam\neq0.$ Use Lemma \ref{symm}($iii$) together with the fact that $f\fm$ is nondegenerate on $\hh$ to fix $x\in\gg_{\bar i}^\a,y\in \gg_{\bar i}^{-\a}$ with $f(x,y)=1$ and $[x,y]\in \hh.$ Take $t_\a$ to be the unique element of $\hh$ representing $\a$ through $f\fm.$ We have
\begin{eqnarray*}
[x\ot t^\lam,\theta(\lam,-\lam)^{-1}y\ot t^{-\lam}]=(t_\a\ot 1)+\sum_{i\in I}d_i(\lam)\lam_i= (t_\a\ot 1)+\lam\in \fh\setminus\{0\}.
\end{eqnarray*}
It follows that  $(\fl,\fh,\fm)$ is an extended affine Lie superalgebra with root system $\mathfrak{R}.$
\hfill$\diamondsuit$}
\end{Example}

For a unital associative algebra $\aa$ and nonempty index supersets $I,J,$ by an $I\times J$-matrix
with entries in $\aa,$ we mean a map $A:I\times J\longrightarrow \aa.$ For $i\in I,j\in J,$ we set
$a_{ij}:=A(i,j)$ and call it {\it $(i,j)$-th entry} of $A.$ By a convention, we denote the matrix
$A$ by $(a_{ij}).$  We also denote the set of all $I\times J$-matrices  with entries in $\aa$ by
$\aa^{I\times J}$ and note that it is a vector superspace, under the componentwise summation and scalar product, with $$\aa^{I\times J}_{\bar i}:=\{A\in \aa^{I\times J}\mid A(I_{\bar t}\times J_{\bar s} )=0;\; \hbox{ $t,s\in\{0,1\}$ with $\bar t+\bar s=\bar i+\bar 1$}\},$$ for $i=0,1.$
 If $A=(a_{ij})\in\aa^{I\times J}$ and
$B=(b_{jk})\in \aa^{J\times K}$ are such that for all $i\in I$ and $k\in K,$ at most for finitely many $j\in J,$ $a_{ij}b_{jk}$'s are nonzero, we define the product $AB$ of $A$ and $B$ to be the
$I\times K$-matrix $C=(c_{ik})$ with $c_{ik}:=\sum_{j\in J}a_{ij}b_{jk}$ for all $i\in I,k\in K.$
We note that if $A,B,C$ are three matrices such that $AB,$ $(AB)C,$ $BC$ and $A(BC)$ are defined,
then $A(BC)=(AB)C.$
 For $i\in I,j\in J$ and $a\in \aa,$ we define $E_{ij}(a)$ to be  the matrix in $\aa^{I\times J}$ whose $(i,j)$-th entry is
`` $a$ " and other entries are zero and if $\aa$ is unital, we set $$e_{i,j}:=E_{i,j}(1).$$  Take $M_{I\times J}(\aa)$ to be the subsuperspace of $\aa^{I\times J}$ spanned by
$\{E_{ij}(a)\mid i\in I,j\in J,a\in A\};$ in fact  $M_{I\times J}(\aa)$  is a superspace with $M_{I\times J}(\aa)_{\bar i}=\hbox{span}_\bbbf\{E_{r,s}(a)\mid |r|+|s|=\bar i\},$ for $i=0,1.$ Also with respect to the  multiplication of matrices, the vector superspace
$M_{I\times I}(\aa)$ is an associative $\bbbf$-superalgebra and so is a Lie superalgebra under the Lie bracket
$[A,B]:=AB-(-1)^{|A||B|}BA$ for all $A,B\in M_{I\times I}(\aa).$ We denote  this Lie superalgebra  by
$\mathfrak{pl}_I(\mathcal{A}).$
 For an element $X\in \mathfrak{pl}_I(\mathcal{A}),$ we set $str(X):=\sum_{i\in I}(-1)^{|i|}x_{i,i}$ and call it the {\it supertrace} of $X.$ We finally  make a convention that if $I$ is a disjoint union of nonempty  subsets
$I_1,\ldots, I_t$ of $I,$ then for an $I\times I$-matrix $A,$ we write
$$A=\left [\begin{array}{ccc} A_{1,1}&\cdots& A_{1,t}\\
A_{2,1}&\cdots&A_{2,t}\\
\vdots&\vdots&\vdots\\
A_{t,1}&\cdots&A_{t,t}\\
\end{array}\right ]$$
in which for $1\leq r,s\leq t,$ $A_{r,s}$ is an $I_r\times I_s$-matrix whose $(i,j)$-th entry
coincides with $(i,j)$-th entry of $A$ for all $i\in I_r,j\in I_s.$ In this case, we say that $A\in \aa^{I_1\uplus\cdots\uplus I_t}$ and note that the defined matrix product  obeys the product of block matrices.
\smallskip

In the next example, we realize a certain extended affine Lie superalgebra using an ``{\it affinization}'' process. To this end, we need to fix some notations. Suppose that   $A$ is a unital associative algebra and  ``$\;*\;$'' is an involution on $\aa$ that is a self-inverting linear endomorphism of $A$ with $(ab)^*=b^*a^*$ for all $a,b\in A.$  We next assume $\dot I,\dot J,\dot K$ are  nonempty  index sets with disjoint copies $\bar {\dot I}=\{\bar i\mid i\in \dot I\},$ $\bar {\dot J}=\{\bar j\mid j\in \dot J\}$ and $\bar {\dot K}=\{\bar k\mid k\in \dot K\}$ respectively. Suppose that $0,0',0''$ are three  distinct symbols and by a convention,  take  $\bar0:=0,$ $\bar0':=0'$ and $\bar0'':=0''.$ We take $I$ to be either $\dot I\uplus\bar{\dot I}$ or $\{0\}\uplus \dot I\uplus\bar{\dot I},$ $J$ to be either $\dot J\uplus\bar{\dot J}$ or $\{0'\}\uplus\dot J\uplus\bar{\dot J},$ and  $K$ to be either $\dot K\uplus\bar{\dot K}$ or $\{0''\}\uplus \dot K\uplus\bar{\dot K}.$  For a matrix $X=(x_{ij})\in M_{I\times J}(\aa),$ define $X^{\diamond}$  to be  the matrix $(y_{ji})$ of $M_{J\times I}(\aa)$ with $y_{ji}:= x_{\bar i\bar j}^*$ ($i\in I,j\in J$) where for an index  $t\in I\cup J,$ by $\bar{\bar t},$ we mean $t.$
It is immediate that for a matrix $X=(x_{ij})\in M_{I\times I}(\aa),$ \begin{equation}\label{dia1}tr(X^{\diamond})=(tr(X))^*.\end{equation} Also if   $X=(x_{ij})\in M_{I\times J}(\aa)$ and $Y\in M_{J\times K}(\aa),$ then for $i\in I$ and $k\in K,$ we have
\begin{eqnarray*}(XY)^\diamond_{ki}=(\sum_{j\in J}x_{\bar ij}y_{j\bar k})^*=\sum_{j\in J} y_{j\bar k}^* x_{\bar ij}^*=\sum_{j\in J} y_{\bar j\bar k}^* x_{\bar i\bar j}^*=\sum_{j\in J} Y^\diamond_{kj} X^\diamond_{ji}=(Y^\diamond X^\diamond)_{ki}\end{eqnarray*} which implies that \begin{equation}\label{dia2}(XY)^\diamond=Y^\diamond X^\diamond.\end{equation}

 \begin{Example}

 {\rm
In this example, we assume that the field  $\bbbf$ contains a forth primitive  root of unity and call it $\zeta.$ Suppose that $G$ is a torsion free additive abelian group and  $\lam$  is a commutative $2$-cocycle satisfying $\lam(0,0)=1.$ Take  $\aa$ to be the commutative   associative torus corresponding to $(G, \lam).$ Take $*$ to be an involution of $\aa$ mapping $\aa^\tau$ to $\aa^\tau$ for all $\tau\in G$ and suppose that $I$ and $J$ are as in the previous paragraph such that  $I\cap J=\emptyset$ and that  $|I|\neq|J|$ if $I$ and $J$ are both finite.
%
Consider $I\uplus J$ as  a superset with $|i|:=\bar 0$ and $|j|:=\bar 1$ for $i\in I$ and $j\in J$ and take $\LL:=\mathfrak{pl}_{I \uplus J}(\aa).$  One knows that for ${}_{[\tau]}\LL:=\{(x_{ij})\in \LL\mid x_{ij}\in\aa^\tau;\;\;\forall i,j\in I\uplus J\}$ ($\tau\in G$),
 ${\displaystyle\LL=\bigoplus_{\tau\in G}{}_{[\tau]}\LL}$ is a $G$-graded Lie superalgebra.
Set $$\gg:=\mathfrak{sl}_\aa(I,J):=\{A\in \mathfrak{pl}_{I\uplus J}(\aa)\mid str(A)=0\}\simeq \frak{sl}_\bbbf(I,J)\ot \aa.$$ As $ \gg$ is a subsuperalgebra of $\LL$ generated by $\{E_{i,j}(a)\mid i,j\in I\uplus J, i\neq j, a\in\aa\},$ it follows that  $\gg$ is a $G$-graded subsuperalgebra of $\LL.$
Setting  $\hh:=\hbox{span}_\bbbf\{e_{i,i}-e_{r,r},e_{j,j}-e_{s,s},e_{i,i}+e_{j,j}\mid i,r\in I, j,s\in J\},$
one knows   that  $\gg$ has a weight space decomposition $\gg=\op_{\a\in \hh^*}\gg^\a$ with respect to $\hh$ with the corresponding root system $$R=\{\ep_i-\ep_r,\d_j-\d_s,\ep_i-\d_j,\d_j-\ep_i\mid i,r\in I, j,s\in J\},$$ where  for $t\in I$ and $p\in J,$
$$\begin{array}{l}\ep_t:\hh\longrightarrow \bbbf;\;\;\; e_{i,i}-e_{r,r}\mapsto \d_{i,t}-\d_{r,t},e_{j,j}-e_{s,s}\mapsto 0,e_{i,i}+e_{s,s}\mapsto \d_{i,t},\\
\d_p:\hh\longrightarrow \bbbf;\;\;\; e_{i,i}-e_{r,r}\mapsto0,e_{j,j}-e_{s,s}\mapsto  \d_{j,p}-\d_{p,s},e_{i,i}+e_{j,j}\mapsto \d_{p,j},
\end{array} $$ $(i,r\in I,\;\; j,s\in J)$
and for $ i,r\in I$ and $ j,s\in J$ with $i\neq r$ and $j\neq s,$ we have
$$
\begin{array}{ll}
\gg^{\ep_i-\ep_r}=\aa e_{i,r},\;\; \gg^{\d_j-\d_s}=\aa e_{j,s},\;\;
\gg^{\ep_i-\d_j}=\aa e_{i,j},\;\;\gg^{\d_j-\ep_i}=\aa e_{j,i},\\
\displaystyle{\gg^0=\{A=\sum_{t\in I\uplus J}a_{tt}e_{t,t}\in \frak{pl}_{I\uplus J}(\aa)\mid str(A)=0\}.}
\end{array}
$$  For $\a\in R$ and $\tau\in G,$ setting  $${}_{[\tau]}\gg^\a:={}_{[\tau]}\gg\cap\gg^\a,$$ we have
$$
\begin{array}{ll}
{}_{[\tau]}\gg^{\ep_i-\ep_r}=\aa^\tau e_{i,r},\;\; {}_{[\tau]}\gg^{\d_j-\d_s}=\aa^\tau e_{j,s},\;\;
{}_{[\tau]}\gg^{\ep_i-\d_j}=\aa^\tau e_{i,j},\;\;{}_{[\tau]}\gg^{\d_j-\ep_i}=\aa^\tau e_{j,i},\\
\displaystyle{{}_{[\tau]}\gg^0=\{A=\sum_{t\in I\uplus J}a_{tt}e_{t,t}\in \frak{pl}_{I\uplus J}(\aa)\mid a_{tt}\in\aa^\tau\;(t\in I\uplus J), str(A)=0\}}
\end{array}
$$ for $ i,r\in I$ and $ j,s\in J$ with $i\neq r$ and $j\neq s.$

Now take $\ep:\aa\longrightarrow \bbbf$ to be a linear function defined by  $$x^\tau\mapsto \left\{\begin{array}{ll}0&\hbox{if $\tau\neq0$}\\
1&\hbox{if $\tau=0$}\end{array}\right.\;\;\;\; (\tau\in G). $$ Define   $$\fm:\gg\times\gg\longrightarrow\bbbf;\;\;(x,y)\mapsto \ep(str(xy)) .$$
This defines  a nondegenerate  invariant even supersymmetric bilinear  form on $\gg.$

Next take $\v:=\bbbf\ot_\bbbz G.$ Since $G$ is torsion free, we can identify $G$ with a subset of $\v$ in a usual manner. We next   fix a basis $B:=\{\tau_t\mid t\in T\}\sub G$ of $\v.$ Suppose that $\{d_t\mid t\in T\}$ is its dual basis  and  take $\v^\dag$ to be  the restricted dual of $\v$ with respect to this basis. Each $d\in \v^\dag$ can be considered as a derivation of $\gg$  (of degree $0$) by $d(x):=d(\tau)x$ for each $x\in {}_{[\tau]}\gg$ ($\tau\in G$).
 Set  $$\fl:=\gg\op\v\op\v^\dag.$$
We  extend the form on
$\gg$
  to a nondegenerate   even supersymmetric bilinear form $\fm$ on $\fl$ by
\begin{equation*}\label{formloop2}
\begin{array}{c}
{\small (\v,  \v)=(\v^\dag,\v^\dag)=(\v,\gg)=(\v^\dag,\gg):=\{0\}\andd (d,v):=d(v)\quad  (d\in \v^\dag, v\in \v).}
\end{array}
\end{equation*}

We also  define
\begin{equation*}\label{bracketloop22}\begin{array}{l}
  \;[d,x]=-[x,d]=d(x),  \quad (d\in\v^\dag,x\in\gg)\\
 \;[\v,\fl]=\{0\},\\
 \;[\v^\dag,\v^\dag]=\{0\},\\
\;[x,y]=[x,y]_{_{\gg}}+\sum_{t\in T}(d_t(x),y)\tau_t, \quad
(x,y\in \gg) \end{array}\end{equation*}
in which $[\cdot,\cdot]_{_\gg}$ is the bracket on $\gg.$ Setting $\fh:=\hh\op\v\op\v^\dag,$ as in the previous example, one gets that   $(\fl,\fh,\fm)$  is an extended affine  Lie superalgebra
with root system $$\mathfrak{R}=\{\a+\tau\mid \a\in R,\tau\in G\}$$ in which  $\a\in R$ is considered as an element of $\fh^*$ by $$\a({\v\op\v^\dag}):=\{0\},$$ and $\tau\in \v$ is considered  as an element of $\fh^*$ by $$\tau(\hh\op\v):=\{0\}\andd \tau(d):=d(\tau)\;\;\; (d\in\v^\dag).$$ We also have
$$\begin{array}{l}
\fl^0=\fh\andd \fl^{\a+\tau}={}_{[\tau]}\gg^\a \;\;\;(\a\in R,\tau\in G \hbox{ with $\a+\tau\neq0$}).
\end{array}$$
Next for $A=\left(\begin{array}{ll}X&Y\\Z&W\end{array}\right)\in\LL=\frak{pl}_{I\uplus J}(\aa),$ define $A^\#:=\left(\begin{array}{ll}-X^\diamond&Z^\diamond\\-Y^\diamond&-W^\diamond\end{array}\right).$ We have $[A,B]^\#=[A^\#,B^\#]$ and so $\#$ is a Lie superalgebra automorphism  of $\LL$ of order 4. Since $\#$ maps $\gg$  to $\gg$ (see (\ref{dia1})),  we consider $\#$ as a Lie superalgebra automorphism of $\gg$ as well.
Suppose that $M=\left(\begin{array}{ll}X&Y\\Z&W\end{array}\right),N=\left(\begin{array}{ll}A&B\\C&D\end{array}\right)$ are elements of $\gg,$
then we have
{\small
\begin{eqnarray}
(M^\#,N^\#)&=&\ep(str (M^\#N^\#))\nonumber\\&=&\ep(tr (X^\diamond A^\diamond-Z^\diamond B^\diamond)-tr(-Y^\diamond C^\diamond+ W^\diamond D^\diamond))\nonumber\\
&\stackrel{(\ref{dia2})}{=}&\ep(tr ((AX)^\diamond +(CY)^\diamond- (BZ)^\diamond -(DW)^\diamond))\nonumber\\
&=& \ep(tr ((AX)^\diamond) +tr ((CY)^\diamond)- tr ((BZ)^\diamond) -tr ((DW)^\diamond))\nonumber\\
&\stackrel{(\ref{dia1})}{=}&\ep((tr (AX))^*) +(tr (CY))^*)- (tr (BZ))^*) -(tr (DW))^*))\label{equality}\\
&=&\ep(tr (AX)  +tr (CY)- tr (BZ) -tr (DW))\nonumber\\
&=&\ep(tr (XA)+tr (YC)- tr (ZB) -tr (WD))\nonumber\\
&=& \ep(str(MN))=(M,N).\nonumber
\end{eqnarray}}
We also have \begin{equation}\label{equality1} d(x^\#)=(d(x))^\#;\;\; d\in \v^\dag,\; x\in \gg.\end{equation}
Now extend $\#$ to  $\fl$ by $v^\#:=v$ for $v\in\v\op\v^\dag.$ It follows from (\ref{equality}) and (\ref{equality1}) that    $\#$ is an automorphism of $\fl$ of order $4$ mapping  $\fh$ to $\fh.$
So we have $\fl=\op_{i=0}^3{}^{[i]}\fl,$ where for $i\in\bbbz,$  $${}^{[i]}\fl:=\{x\in\fl\mid x^\#=\zeta^i x\}$$  in which $[i]$ indicates the congruence of $i\in\bbbz$ modulo $4\bbbz.$ Using  (\ref{equality}) together with  the fact that the form on $\fl$ is nondegenerate, for $i,j\in\bbbz,$ we have \begin{equation}\label{form}({}^{[i]}\fl,{}^{[j]}\fl)\neq \{0\}\hbox{ if and only of $i+j\in 4\bbbz.$}\end{equation}
Next take $\sg$ to be the restriction of $\#$ to $\fh$ and  set $\mathfrak{h}^\sg$ to be the set of fixed points of $\fh$ under $\sg.$ Consider  a linear endomorphism of $\fh^*$ mapping $\a\in\fh^*$ to $\a\circ\sg^{-1}$ and denote it again by $\sg.$ The Lie superalgebra $\fl$ has a weight space decomposition
$\fl=\sum_{\{\pi(\a)\mid\a\in \frak{R}\}}\fl^{\pi(\a)}$ with respect to $\fh^\sg$ where  $$\pi(\a):=(1/4)(\a+\sg(\a)+\sg^2(\a)+\sg^3(\a))\;\;\; (\a\in \frak{R})$$ (see \cite[(2.11) \& Lem. 3.7]{ahy}). Moreover, we have
\begin{eqnarray*} \pi(\frak{R})&=&\{\pi(\a)\mid \a\in\mathfrak{R}\}\\
&=&\{\tau\mid \tau\in G\}\\
&\cup&\{\frac{1}{2}((\ep_i-\ep_r)+(\ep_{\bar r}-\ep_{\bar i}))+\tau\mid  \tau\in G,i,r\in I; i\neq r\}\\
&\cup& \{\frac{1}{2}((\d_j-\d_s)+(\d_{\bar s}-\d_{\bar j}))+\tau\mid \tau\in G, j,s\in J; j\neq s\}\\
&\cup&
 \{\frac{1}{2}((\ep_i-\d_j)+(\d_{\bar j}-\ep_{\bar i}))+\tau\mid \tau\in G, i\in I,j\in J\}\\
 &\cup& \{\frac{1}{2}((\d_j-\ep_i)+(\ep_{\bar i}-\d_{\bar j})+\tau\mid \tau\in G, i\in I, j\in J\}\\
 &=&\{\tau\mid \tau\in G\}\\
&\cup&
 \{\frac{1}{2}((\ep_i-\ep_{\bar i})-(\ep_{ r}-\ep_{\bar r}))+\tau\mid \tau\in G, i,r\in I; i\neq r\}\\
 &\cup&
 \{\frac{1}{2}((\d_j-\d_{\bar j})-(\d_{s}-\d_{\bar s}))+\tau\mid \tau\in G, j,s\in J; j\neq s\}\\
 &\cup&
 \{\frac{1}{2}((\ep_i-\ep_{\bar i})-(\d_{ j}-\d_{\bar j}))+\tau\mid \tau\in G, i\in I, j\in J\}\\
 &\cup&
 \{\frac{1}{2}((\d_{ j}-\d_{\bar j})-(\ep_i-\ep_{\bar i}))+\tau\mid \tau\in G, i\in I, j\in J\}
 \end{eqnarray*}
which is an extended affine root supersystem of type  $BC(I,J)$ if   $\{0,0'\}\cap (I\cup J)\neq \emptyset,$ and it  is of type $C(I,J),$ otherwise; see \cite{you6} for the notion of type for an extended affine root supersystem.
We next note that   for $i\in\{0,1,2,3\},$ as  ${}^{[i]}\fl$ is an $\fh^\sg$-submodule of $\fl,$ it  inherits  the   weight space decomposition ${\displaystyle{}^{[i]}\fl=\sum_{\{\pi(\a)\mid\a\in \frak{R}\}}{}^{[i]}\fl^{\pi(\a)}}$ from $\fl$ in which ${}^{[i]}\fl^{\pi(\a)}:={}^{[i]}\fl\cap\fl^{\pi(\a)}.$  We recall  (\ref{form}) together with the fact that  the form $\fm$ is nondegenerate. So if $\a,\b\in \frak{R}$ and $i,j\in\bbbz,$ we get that \begin{equation}\label{form1}
\begin{array}{l}\hbox{ for $0\neq x\in{}^{[i]}\fl^{\pi(\a)},$} (x,{}^{[j]}\fl^{\pi(\b)})\neq \{0\}\\
\hbox{ if and only if  $i+j\in 4\bbbz$ and $\pi(\a)+\pi(\b)=0.$}
\end{array}\end{equation}
Now we set $$\tilde\fl:=\sum_{i\in\bbbz}({}^{[i]}\fl\ot\bbbf t^i)\op\bbbf c\op\bbbf d$$ where $c,d$ are two symbols. Since $\#$  preserves the $\bbbz_2$-grading on $\fl,$ $\tilde\fl$ is a superspace with $$\tilde\fl_{\bar 0}:=\sum_{i\in\bbbz}(({}^{[i]}\fl\cap \fl_{\bar 0})\ot\bbbf t^i)\op\bbbf c\op\bbbf d\andd \tilde\fl_{\bar 1}:=\sum_{i\in\bbbz}(({}^{[i]}\fl\cap \fl_{\bar 1})\ot\bbbf t^i).$$ Moreover, $\tilde\fl$ together with the following bracket \\

$\begin{array}{ll}
[x\ot t^i+rc+sd,y\ot t^j+r'c+s'd]^{\;\tilde{}}&:=[x,y]\ot t^{i+j}+i\d_{i,-j}(x,y)c\\
&+sj y\ot t^j-s'ix\ot t^i
\end{array}$

\noindent is a Lie superalgebra equipped with   a weight space decomposition with respect to $(\mathfrak{h}^\sg\ot \bbbf)\op\bbbf c\op\bbbf d.$ More precisely, if we  define
$$\d:(\fh^\sg\ot\bbbf)\op\bbbf c\op\bbbf d\longrightarrow \bbbf;\;\;c\mapsto0,d\mapsto1,h\ot1\mapsto 0\;\;(h\in\fh^\sg),$$ then we get that  $\tilde{\mathfrak{R}}:=\{\pi(\a)+i\d\mid \a\in \mathfrak{R} ,i\in\bbbz,\;{}^{[i]}\fl\cap \fl^{\pi(\a)}\neq\{0\}\}$ is the corresponding  root system of $\tilde\fl$ in which $\pi(\a)$ is considered as an element of the dual space of $(\fh^\sg\ot\bbbf)\op\bbbf c\op\bbbf d$  mapping $h\ot 1\in \fh^\sg\ot\bbbf$  to $\a(h)$ and $c,d$ to $0.$ Furthermore, since for each $\a\in \frak{R}\setminus\{0\},$ $\pi(\a)\neq 0,$ as in \cite[Cor. 3.26]{ABP}, we have $ {}^{[0]}\fl^{\pi(0)}=\fh^\sg.$ So   for $\a\in \mathfrak{R}$ and $i\in\bbbz,$ we have 

$$\tilde\fl^{\pi(\a)+i\d}=\left\{\begin{array}{ll}
({}^{[i]}\fl\cap \fl^{\pi(\a)})\ot t^i& \hbox{if }(\a,i)\neq(0,0)\\
(\fh^\sg\ot 1)\op \bbbf c\op\bbbf d& \hbox{if } (\a,i)=(0,0).\\
\end{array}
\right.$$
We  extend the form on
$\fl$
  to a  supersymmetric  bilinear form $\fm^{\tilde{}}$ on $\tilde\fl$ by
\begin{equation}\label{formloop}
\begin{array}{l}
(c,d)^{\tilde{}}=1,\;(c,c)^{\tilde{}}=(d,d)^{\tilde{}}=(c,{}^{[i]}\fl\ot\bbbf t^i)^{\tilde{}}=(d,{}^{[i]}\fl\ot\bbbf t^i)^{\tilde{}}:=\{0\}\;\;(i\in\bbbz)\\
(x\ot t^i,y\ot t^j)^{\tilde{}}=\d_{i+j,0}(x,y)\;\; (i,j\in\bbbz,x\in{}^{[i]}\fl,y\in {}^{[j]}\fl).
\end{array}
\end{equation}
Since the form on $\fl$ is even and nondegenerate, $\fm^{\tilde{}}$ is also even and nondegenerate; moreover, by (\ref{form1}), if  $j\in\{0,1\},$  $\a\in \mathfrak{R}$ and $i\in\bbbz$ with $(i,\a)\neq(0,0)$ such that ${}^{[i]}\fl^{\pi(\a)}\cap \fl_{\bar j} \neq\{0\},$ we have $({}^{[i]}\fl^{\pi(\a)}\cap \fl_{\bar j},{}^{[-i]}\fl^{\pi(-\a)}\cap \fl_{\bar j})\neq \{0\}.$  Using this together with the fact that $ {}^{[0]}\fl^{\pi(0)}=\fh^\sg$ and Lemma \ref{symm}($iii$), one finds $x\in {}^{[i]}\fl^{\pi(\a)}\cap \fl_{\bar j},y\in {}^{[-i]}\fl^{\pi(-\a)}\cap \fl_{\bar j}$ with $0\neq [x,y]\in \fh^\sg.$ So $$[x\ot t^i,y\ot t^{-i}]^{\;\tilde{}}=[x,y]\ot 1+i(x,y)c\in (\fh^\sg\ot 1)\op \bbbf c\op\bbbf d\setminus\{0\}.$$ Also as $\pi({\mathfrak{R}})$ is an extended affine root supersystem, it follows that $ad_x$ is locally nilpotent for all $x\in\tilde\fl^{\tilde\a}$ where $\tilde\a\in {\tilde{\mathfrak{R}}}_{re}^\times.$
So   $$(\tilde\fl,(\fh^\sg\ot \bbbf)\op\bbbf c\op\bbbf d,\fm^{\tilde{}})$$ is an extended affine Lie superalgebra with corresponding root system $\tilde{\mathfrak{R}}.$ \hfill{$\diamondsuit$}}
\end{Example}

\centerline{Acknowledgement}
This research was in part
supported by a grant from IPM (No. 92170415) and partially carried out in IPM-Isfahan branch.

\end{document}